\documentclass{amsart}

\usepackage{dsfont}
\usepackage{amsmath}
\usepackage{amssymb}
\usepackage{graphicx}

\usepackage{amssymb}
\usepackage{comment}

\usepackage{amsfonts}

\usepackage{soul}

\usepackage{mathpazo}
\usepackage{color}
\usepackage{yfonts}

\usepackage{paralist}

\usepackage{stmaryrd}

\usepackage{amsxtra}

\def\cA{\mathcal A}

\newcommand{\cL}{\mathcal L}

\newcommand{\norm}[1]{\| #1\|}

\def\cA{          \mathcal A}

\def\cC{          \mathcal C}
\def\cD{          \mathcal D}

\def\Eu{           E^{u}}
\def\Es{           E^{s}}
\def\Ec{           E^{c}}

\def\cWu{          \mathcal W^{u}}
\def\cWs{          \mathcal W^{s}}

\def\clb{   \color{black}}

\let\cal\mathcal

\def \R{{\mathbb R}}

\def \Z{{\mathbb Z}}
\def \N{{\mathbb N}}

\newcommand{\T}{{\mathbb T}}
\newcommand{\prf}{{\begin{proof}}}
\newcommand{\epf}{{\end{proof}}}

\newcommand{\EV}{{\mathbb E}}
\newcommand{\PP}{{\mathbb P}}

\newcommand{\GG}{{\mathbb G}}

\DeclareMathOperator{\diff}{Diff}

\newtheorem{theo}{\sc Theorem}
\newtheorem{prop}{\sc Proposition}

\newcommand{\ary}{\begin{eqnarray}}
\newcommand{\eary}{\end{eqnarray}}

\newcommand{\aryst}{\begin{eqnarray*}}
\newcommand{\earyst}{\end{eqnarray*}}

\newcommand{\enmt}{\begin{enumerate}}
\newcommand{\eenmt}{\end{enumerate}}

\newtheorem{lemma}{\sc lemma}

\newtheorem{cor}{\sc corollary}

\newtheorem{conj}[theo]{\sc Conjecture}

\theoremstyle{definition}

\def\bee{\begin{equation}}
\def\eee{\end{equation}}

\newtheorem{defi}{\sc Definition}

\newtheorem{nota}{\sc Notation}

\theoremstyle{rema}

\newtheorem{rema}{\sc Remark}

\newcommand{\Diff}{\text{Diff}}

\newcommand{\pdvr}[2]
{\dfrac{\partial^{#2} #1}{\partial \theta^{#2_1} \partial r^{#2_2}}}
\newcommand{\pdvrs}[2]
{\partial^{#2} #1 /\partial \theta^{#2_1} \partial r^{#2_2}}
\newtheorem{thm}{Theorem}

\numberwithin{equation}{section}

\author{Zhiyuan Zhang}
\begin{document}

\title[Stable transitivity of group actions]{On stable transitivity of finitely generated groups of volume preserving diffeomorphisms}

\date{\today}

\maketitle

\begin{abstract} 
In this paper, we provide a new criterion for the stable transitivity of volume preserving finite generated group on any compact Riemannian manifold. As one of our applications, we generalised a result of Dolgopyat and Krikorian in \cite{DK} and obtained stable transitivity for random rotations on the sphere in any dimension. As another application, we showed that for $\infty \geq r \geq 2$, any $C^r$ volume preserving partially hyperbolic diffeomorphism $g$ on any compact Riemannian manifold $M$ having sufficiently H\"older stable or unstable distribution, for any sufficiently large integer $K$, for any $(f_i)_{i=1}^{K}$ in a $C^1$ open $C^r$ dense subset of $\textnormal{\Diff}^r(M,m)^K$, the group generated by $g, f_1,\cdots, f_K$ acts transitively.
\end{abstract}

\tableofcontents
\addtocontents{toc}{\protect\setcounter{tocdepth}{1}}

\section{Introduction}

Let $(M, d_M)$ be a smooth $d-$dimensional compact  Riemannian manifold endowed with the volume form $\mu$ given by the Riemannian metric.
We denote by $Homeo(M)$ ( resp. $Homeo(M, \mu)$ ) the set of ( resp. volume preserving ) homeomorphisms of $M$.
For any $r \in \R_{\geq 1} \bigcup \{\infty\}$, we denote by $\Diff^{r}(M)$ ( resp. $\Diff^{r}(M,\mu)$ ) the set of $r-$differentiable ( resp. volume preserving ) diffeomorphisms of $M$. 
Given an integer $m \geq 2$ and $m$ diffeomorphisms $f_1, \cdots, f_m \in \Diff^{1}(M)$, we are interested in the dynamics of the iterations of these $m$ diffeomorphisms.

Following the definitions in the previous literatures, we have the following.
\begin{defi}\label{defi basic stuff}
For any integer $m \geq 2$, $f_1,\cdots, f_m \in Homeo(M)$, we say that $\{f_i\}_{i=1}^{m}$ is an \textit{IFS (i.e. iterated function system)}.  
We say that $\{f_i\}_{i=1}^m$ is \textit{transitive} or equivalently, \textit{the semigroup generated by $\{f_i\}_{i=1}^m$ acts transitively} if there exists a point $x \in M$ such that for any open set $U \subset M$, there exists an integer $\ell \geq 1$, and a finite word $\omega = \omega_1\omega_2 \cdots \omega_{\ell} \in \{ 1,\cdots, m\}^{\ell}$ such that $f_{\omega_{\ell}} \cdots f_{\omega_1}(x) \in U$. If $\{f_i\}_{i=1}^m$ preserve a common invariant measure $\nu$, we say that $\{f_i\}_{i=1}^{m}$ is \textit{ergodic} with respect to $\nu$ if any $\{f_i\}_{i=1}^m-$invariant Borel set $A \subset M$ has $\nu-$measure $0$ or $1$. 
\end{defi}

Just as that any ergodic volume preserving homeomorphism on $M$ is transitive, any IFS $\{ f_i \}_{i=1}^{ m} \subset Homeo(M, \mu)$ is transitive if $\{f_i\}_{i=1}^{m}$ is ergodic with respect to $\mu$.

A general paradigm in the study of iterations of multiple diffeomorphisms is that transitivity, or even ergodicity is fairly generic in the volume preserving setting. This paper is an attempt to address this issue. We will mainly focus on volume preserving IFS.

\begin{defi}
For any  $1 \leq r \leq s \leq \infty$, an IFS $\{f_i\}_{i=1}^{m} \subset \textnormal{\Diff}^{s}(M, \mu)$ is said to be $C^r-$\textit{stably transitive} ( resp. $C^r-$\textit{stably ergodic} with respect to $\mu$ ) \textit{ in $\diff^{s}(M, \mu)$} if any IFS $\{g_i\}_{i=1}^{m} \subset  \Diff^{s}(M,\mu)$ such that  $g_i $ is sufficiently $C^r-$close to $f_i$ for all $1 \leq i \leq m$ is transitive ( resp. ergodic with respect to $\mu$ ).
\end{defi}

Precise conjectures are formulated in \cite{DK}, \cite{KN}.
\begin{conj}[in \cite{DK}]\label{conj 1}
Let $M$ be a compact Riemannian manifold and $r$ be a suffciently large number. Let $m \geq 2$ be an integer. Then $C^r-$stable ergodicity is open and dense among $m-$tuples of $C^r$ volume preserving diffeomorphisms of $M$.
\end{conj}
\begin{conj}[in \cite{KN}]\label{conj 2}
For any $r \in \N \bigcup \{\infty\}$, a $C^r$ generic pair of volume-preserving diffeomorphisms generates  a transitive iterated function system.
\end{conj}

It is straightforward to see that the transitivity ( resp. ergodicity ) of an IFS is implied by the transitivity ( resp. ergodicity ) of one of the maps in the IFS. The converse implication is however false : one can take two translations on $\T^2$, defined by $f_1(x,y) = (x+\alpha, y)$, $f_2(x,y) = (x, y+\beta)$, with $\alpha,\beta$ rationally independent. Then both $f_1,f_2$ are not transitive, while their product $f_1f_2$ is ergodic with respect to the volume. 

In both conjectures, we do not require the diffeomorphisms to be orientation preserving. When $M$ is a circle with Lebesgue measure, the above conjectures are easy exercises. However, when $M$ equals to any given surface, the above conjectures already require considerable new ideas ( see \cite{KN} ). To the best of the author's knowledge, Conjecture \ref{conj 1} is open for any manifold of dimension at least $2$, and Conjecture \ref{conj 2} is open for any manifold of dimension at least $3$. 

In \cite{DK}, the authors showed that any IFS $\{f_i\}_{i=1}^m$ consisted of $C^{\infty}$ volume preserving diffeomorphisms of the $d-$dimensional sphere $\mathbb{S}^d$ ( $d \geq 2$ )  which are $C^r$ close to  a set of topological generators of $\mathbb{SO}_{d+1}$  for $r \gg 1$, can be simultanenously conjugate to rotations by a $C^{\infty}-$diffeomorphism, if there exists at least one stationary measure  for $\{f_i\}_{i=1}^m$  with vanishing Lyapunov exponents. As a consequence, they showed the following interesting stable ergodicity result.
\begin{thm}[ Corollary 2 in \cite{DK} ]\label{thmDK}
If $d \geq 2$ is even and $R_1,\cdots, R_m$ topologically generate $\mathbb{SO}_{d+1}$, then the system $\{R_{i}\}_{i=1}^{m}$ is $C^r-$stably ergodic with respect to the volume in $\diff^{\infty}(\mathbb{S}^d, Vol)$  for sufficiently large $r \geq 1$. Namely, there exist $r_0 \geq 1$, $\epsilon > 0$ such that if for each $1 \leq i \leq m$, $f_{i} \in \diff^{\infty}(\mathbb{S}^{d}, Vol)$ and $d_{C^{r_0}}(f_{i}, R_{i}) < \epsilon$, then the IFS $\{f_{i}\}_{i = 1}^m$ is ergodic with respect to the volume form.
\end{thm}

Much is known about the topological generators of semisimple Lie groups, such as $\mathbb{SO}_{d}, d \geq 3$. For example, by \cite{Au}, for any $d \geq 3$, pairs generating $\mathbb{SO}_{d}$ form an open and dense subset of $\mathbb{SO}_{d}^2$.

In \cite{KN}, the authors proved Conjecture \ref{conj 2} for compact orientable surfaces. 
\begin{thm}\label{thmKN}
Let $r \in \N \bigcup \{\infty\}$, and let $S$ be a compact orientable surface with a smooth volume form $\mu$. There is a residual set $\cal{R} \subset \textnormal{\Diff}^{r}(S,\mu)\times \textnormal{\Diff}^r(S,\mu)$ in the product $C^r$ topology such that if $(f,g) \in \cal{R}$ then the $IFS$ $\{f,g\}$ is transitive.
\end{thm}

Compared to Theorem \ref{thmDK}, the proof of Theorem \ref{thmKN} is mostly by topological arguments. So far Conjecture \ref{conj 1}, namely the density of stable ergodicity, is still unknown for surface diffemorphisms. We mention another line of research which focus on the classification of ergodic stationary measures for random dynamics. A recent breakthrough in this direction is obtained by Brown-Rodriguez Hertz in \cite{BRH}, for random dynamics of $C^2$ surface diffeomorphisms. 

Besides its intrinsic interest, one can find applications of IFS based arguments in the study of the stable ergodicity conjecture for partially hyperbolic systems, and the instability problem in Hamiltonian systems. We refer the readers to the discussions in \cite{ACW, DK,KN,NP}.

The problem of $C^1$ generic transitivity is known even for a single diffeomorphism. More precisely, Bonatti and Crovisier in \cite{BC} showed that a $C^1$ generic conservative diffeomorphism ( in any dimension) is transitive. The analogous statement is false for $C^r$ with sufficiently large $r$ by $KAM$ theory. This indicates that the validity of the above conjectures would necessarily rely on features specific to multiple diffeomorphisms.

In this paper, we provide a criterion of transitivity ( Propositions \ref{prop main criterion} ) for multiple diffeomorphisms. As one application, 
we partially generalise the result in \cite{DK} to any dimension and obtain the following.
\begin{thm} \label{thm generalization DK}
Given an integer $d \geq 2$, there exists a number $r_0 > 1$ such that for any integer $m \geq 2$, for any set of rotations $R_{1},\cdots, R_{m}$ in $\mathbb{SO}_{d+1}$ such that $R_{1}, \cdots, R_{m}$ topologically generate $\mathbb{SO}_{d+1}$, there exists $\epsilon > 0$ such that any $\{ f_i \}_{i=1}^{m} \subset \diff^{\infty}(\mathbb{S}^d, Vol)$ with $\max_{i}d_{C^{r_0}}(R_{i}, f_{i}) < \epsilon$, any closed $\{f_{i}\}_{i=1}^m-$invariant set has zero or total Lebesgue measure. In particular, the IFS $\{f_{i}\}_{i=1}^m$ is transitive. 
\end{thm}

The proof of Theorem \ref{thm generalization DK} is built upon the works in \cite{DK}. The main technical result in \cite{DK} is a dichotomy, which asserts that under the condition of Theorem \ref{thm generalization DK}, either one can linearise the dynamics to rigid rotations, or there is a \lq\lq uniform splitting \rq\rq for the associated random dynamics ( the formal definition is in Definition \ref{def what is uniform}). A similar dichotomy was obtain by Avila-Viana \cite{AM} for \lq\lq smooth cocycles  \rq\rq over hyperbolic homeomorphisms, such as the shift map on $\{1,\cdots, m\}^{\Z}$ ( see Example 2.1 in \cite{AM} ). As an application of the main result in \cite{AM}, any IFS $\{f_i\}_{i=1}^{m} \subset \diff^{1}(M)$, volume preserving or not, has vanishing Lyapunov exponents for a stationary measure only if such stationary measure is a common invariant measure for $\{f_i\}_{i=1}^m$. However the presence of positive Lyapunov exponents does not in general imply the existence of a uniform splitting, such as the one defined in this paper. In the case of Theorem \ref{thm generalization DK}, such uniform splitting is obtained as a consequence of the \lq\lq rapidly mixing \rq\rq of the joint actions of random rotations on $\mathbb{S}^d, d \geq 2$. We refer the readers to  \cite{DK} Section 4 and \cite{Do2} for more on rapidly mixing dynamics.

In \cite{DK}, the authors exploited the above uniform splitting in even dimensions, in which case there is no zero Lyapunov exponent. By using a Hopf-type argument, they conclude ergodicity. While in odd dimensions, the possibility of having a zero Lyapunov exponent prevented the use of such Hopf-type argument.  The main technical result in our paper shows that such a uniform splitting ( regardless of the dimension ) always imply transitivity.

We note that our criterion is $C^1-$robust in the presence of some additional structure called uniformly dominated splitting ( Definition \ref{def partially hyperbolic} ). As an application, we verify a weaker version of Conjecture 2 in certain subspaces of tuples of multiple diffeomorphisms. 

\begin{thm} \label{thm transitive group}
Let $M$ be a compact Riemannian manifold of dimension $d \geq 2$ with volume form $\mu$ induced by the metric. Let $r \in \R_{\geq 2} \bigcup \{ \infty \}$, and $\theta \in (\frac{d}{d+1},1)$. Let $g \in \textnormal{\Diff}^{r}(M, \mu)$ be a diffeomorphism admitting a uniformly dominated splitting $TM = E_1 \oplus E_2$ ( Definition \ref{def partially hyperbolic} ) such that either $E_1$ or $E_2$ is $\theta-$H\"older ( in particular for any volume preserving partially hyperbolic system $g$ with $\theta-$H\"older stable or unstable distribution  ). Then for all sufficiently large integer $L$, for any $L-$tuple $(f_1,\cdots, f_L)$ in a $C^1$ open $C^r$ dense subset of $\textnormal{\Diff}^{r}(M, \mu)^L$, the IFS $\{ g, f_1,\cdots, f_L \}$ is transitive. Moreover, any $\{ g, f_1,\cdots, f_L \}-$invariant closed set has zero or total volume.
\end{thm}

We note that in Theorem \ref{thm transitive group}, we fix the map $g$ and consider a generic $L-$tuple $(f_1,\cdots, f_L)$. However, if the splitting in Theorem \ref{thm transitive group}, $TM = E_{1} \oplus E_{2}$, is strongly \textit{pinched} ( see Definition \ref{def pin ds} ), the condition in Theorem \ref{thm transitive group} for $g$ is preserved under $C^1$ small perturbation. For the formal statement, see Corollary \ref{cor pinch}.

Theorem \ref{thm transitive group} is consistent with Pugh-Shub's stable ergodicity conjecture which asserts that ergodicity holds in a $C^2$ open and dense subset of the space of $C^2$ volume preserving partially hyperbolic diffeomorphisms. In fact, the validity of stable ergodicity conjecture would imply Theorem \ref{thm transitive group} for $r =2$ and $K = 1$, even without the H\"older condition.

In both Theorem \ref{thm generalization DK}, \ref{thm transitive group}, we can conclude, in loose terms, \lq\lq topological \rq\rq ergodicity, where we replace the measurable sets in the definition of ergodicity by closed sets.
Unfortunately, this extra requirement is indispensable for our method. That is why we fall short of proving ergodicity. Our results would imply ergodicity if we know that any ergodic invariant measure has open basin. However, this seems to be difficult to check under generic conditions. For dissipative systems, one can sometimes conclude ergodicity through studying the topological properties of the basins of ergodic measures. We refer the readers to \cite{BDPP} for a stable ergodicity result using similar ideas.

\subsection{Plan of the paper} 
In Section \ref{preliminary}, we introduce basic notations and results needed for later sections. In this section the main tool is Theorem \ref{LE random dyn} that establishs properties of Lyapunov exponents for random dynamics. Section \ref{main section} contains our main technical result.   We introduce a certain  uniform splitting for random dynamics defined in Subsection \ref{Averaged uniform growth} and state our main criterion Proposition \ref{prop main criterion}. We give a uniform lower bound for the stable manifolds in Subsection \ref{Stable manifolds with uniformly lower bounded size}, then in Subsection \ref{A criteria for metric transitivity} we relate transitivity to this lower bound. In Section \ref{Transversality and growth} we show how to check the conditions in our criterion. We prove our main theorems in Section \ref{Proof of the main theorems}. 

\subsection*{Acknowledgement}
I thank Artur Avila for his encouragement. I also thank Rapha\"el Krikorian for useful conversations and especially for explaining to me some of the ideas in \cite{DK}. I thank Meysam Nassiri for useful conversations and providing related references. I thank Danijela Damjanovic for related discussions. I thank the anonymous referees for many helpful suggestions that greatly improve the presentation of the paper.

\section{Preliminary}\label{preliminary}

In this section, we will recall some basic notions and results in the study of the dynamics of multiple diffeomorphisms which are relevant to our results.

\subsection{Random dynamics}

Given a compact Riemannian manifold $(M, d_{M})$ endowed with volume form $\mu$ induced by the Riemannian metric, an integer $m \geq 2$ and an IFS $\{f_i\}_{i=1}^{m} \subset \textnormal{\Diff}^1(M)$. Let $\Omega  = \{1,\cdots, m\}^{\Z}$ and denote by $\PP$ the Bernoulli distribution on $\Omega$ with uniform distribution on $\{1,\cdots, m\}$, i.e.
\aryst
\PP(\omega | \omega_{-n} = \zeta_{-n}, \cdots, \omega_{n} = \zeta_{n}) = m^{-2n-1}, \forall n \in \N, (\zeta_{-n}, \cdots, \zeta_{n}) \in \{1,\cdots, m\}^{2n+1}.
\earyst

To simply the notations, for any function $\varphi \in L^1(\Omega, d\PP)$, we denote by $\EV(\varphi) = \int \varphi(\omega) d\PP(\omega)$.

For any $n \geq 1$, any  $\omega = (\omega_{i})_{i \in \Z} \in \Omega$, we set
\begin{eqnarray*}
f_{\omega}^{n} = f_{\omega_{n-1}} \cdots f_{\omega_0}.
\end{eqnarray*}
More generally, for any integer $k \in \Z$, $j \in \N$, we set
\begin{eqnarray*}
f^{k,k+j}_{\omega} &=& f_{\omega_{k+j-1}} \cdots f_{\omega_{k}}, \\
f^{k,k-j}_{\omega} &=& f^{-1}_{\omega_{k-j}} \cdots f^{-1}_{\omega_{k-1}}.
\end{eqnarray*}
Denote the shift map on $\Omega$ by 
\begin{eqnarray*}
T_0 : \Omega &\to& \Omega \\
(T_0(\omega))_n &= &\omega_{n+1},
 \end{eqnarray*}
and the suspension map on $\Omega \times M$ by
\begin{eqnarray*}
T : \Omega \times M &\to& \Omega \times M \\
T(\omega, x)&=& (T_0(\omega), f_{\omega_0}(x)).
\end{eqnarray*}

In the following, we will focus on the case where $\{f_i\}$ preserve the volume form $\mu$. In this case, it is clear that the measure $\PP \times \mu$ is $T-$invariant.

Similar to the study of deterministic dynamics i.e. iterations of a single diffeomorphism, we can also define Lyapunov exponents for random dynamics.   Both the statement  and the proof of the following theorem is similar to that of the well-known Oseledec's theorem.

\begin{thm} \label{LE random dyn}  ( Chapter I, Theorem 3.2 in \cite{LQ} )
Given $f_1, \cdots, f_m \in \textnormal{\Diff}^{1}(M,\mu)$.
There exists a measurable set $\Lambda_0 \subset \Omega \times M$ with $(\PP \times \mu)(\Lambda_0) = 1, T(\Lambda_0) = \Lambda_0$ such that :

\begin{enumerate}
\item For every $(\omega,x) \in \Lambda_0$ there exists a decomposition of $T_{x}M$
\begin{eqnarray*}
   T_{x}M = V^{(1)}(\omega, x) \oplus \cdots \oplus V^{(\gamma(x))}(\omega, x)
\end{eqnarray*}
and numbers
\begin{eqnarray*}
-\infty < \lambda^{(1)}( x) < \lambda^{(2)}(x) < \cdots < \lambda^{(\gamma( x))}( x) < +\infty
\end{eqnarray*}
which depend only on $x$, such that 
\begin{eqnarray*}
\lim_{n \to \infty} \frac{1}{n} \log \norm{Df^{n}_{\omega}(x, \xi)} = \lambda^{(i)}(x)
\end{eqnarray*}
for all $\xi \in V^{(i)}(\omega,x) \setminus \{0\} , 1 \leq i \leq \gamma(x)$. Moreover, $\gamma(x), \lambda^{(i)}(x)$ and $V^{(i)}(\omega,x)$ depend measurably on $(\omega,x) \in \Lambda_0$ and 
\begin{eqnarray*}
\gamma(f_{\omega_0}(x)) = \gamma( x), \lambda^{(i)}(f_{\omega_0}(x)) = \lambda^{(i)}(x), Df_{\omega_0}V^{(i)}(\omega,x) = V^{(i)}(T(\omega,x))
\end{eqnarray*}
for each $(\omega,x) \in \Lambda_0, 1 \leq i \leq \gamma(x)$.

\item For each $(\omega,x) \in \Lambda_0$, if $\{ \xi_1, \cdots, \xi_d \}$ is a basis of $T_{x}M$ which satisfies that for each $i=1,\cdots, \gamma(x)$ there are exactly $\dim V^{(i)}$ many indexes $j$ satisfying \begin{eqnarray} \label{basis le}
\lim_{n \to +\infty} \frac{1}{n} \log \norm{Df^{n}_{\omega}(x,\xi_j)} = \lambda^{(i)}(x),
\end{eqnarray}
then for every two non-empty disjoint subsets $P,Q \subset \{1, \cdots, d\}$ we have
\begin{eqnarray} \label{angle le}
\lim_{n \to +\infty} \frac{1}{n} \log |\angle( Df^{n}_{\omega} E_{P}, Df^{n}_{\omega}E_{Q})| = 0
\end{eqnarray}
where $E_{P}$ and $E_{Q}$ denote the subspaces of $T_{x}M$ spanned by the vectors $\{\xi_i\}_{i \in P}$ and $\{\xi_i\}_{i \in Q}$ respectively.
\end{enumerate}

\end{thm}
\begin{rema}\label{rema1}
In Theorem 3.2 \cite{LQ}, the author considered the non-invertible map defined on $\{1,\cdots, m\}^{\N} \times M$ and got a filtration of subspaces. Here we consider an invertible map and get a splitting.
\end{rema}

We call the numbers $\lambda^{(1)}(x), \cdots, \lambda^{(\gamma(x))}(x)$ the \textit{Lyapunov exponents} at $x$.

Now we recall some facts about stable and unstable manifolds of random dynamical systems. We follow the presentations in \cite{DK}. More detailed informations can be found in \cite{LQ}, Chapter III.

Given $f_1, \cdots, f_m \in \Diff^{1+\epsilon}(M,\mu)$ for some $\epsilon > 0$, for each $(\omega, x) \in \Omega \times M$, we define 
\begin{align*}
\cWs_{\omega}(x) = \{ y |  d(f^{n}_{\omega}(x), f^{n}_{\omega}(y) ) \to 0 \mbox{  exponentially fast, }  n \to \infty \}.
\end{align*}
Let $\varrho_{in}$ be the injectivity radius of $M$. We define 
\begin{align*}
\cWs_{\omega, loc}(x) = \{ y \in \cWs_{\omega}(x) |  d(f^{n}_{\omega}(x), f^{n}_{\omega}(y) ) < \varrho_{in}/2, \forall n \geq 0 \}.
\end{align*}
For $(\PP \times \mu)-a.e. (\omega,x)$ such that the Lyapunov exponents defined at $(\omega,x)$ are not all zero, $\cWs_{\omega}(x)$ are $C^{1}-$manifolds. We call $\cWs_{\omega}(x)$ ( resp. $\cWs_{\omega, loc}(x)$) the \textit{stable manifold} (resp. \textit{local stable manifold} ) associated to $(\omega,x)$. Similarly we can define unstable manifolds for $(\PP \times \mu)-a.e. (\omega,x)$, denoted by $\cWu_{\omega}(x)$, by considering $n$ tends to $-\infty$.
Endow $\cWs_{\omega}(x)$ with the induced Remannian distance, and for any $l > 0$ we denote by $\cWs_{\omega}(x,l)$ the $l-$ball in $\cWs_{\omega}(x)$.
We shall use the \textit{absolute continuity} of the lamination $\cWs_{\omega}$ ( for $\PP-a.e. \omega$). More precisely, we have :

$\textbf{(AC)}$ For almost all $\omega$ the following holds.
For any integer $1 \leq k \leq d$, let $\Gamma_k \subset M$ denote
\begin{eqnarray}
  \nonumber
\Gamma_k =& \{  x | (\omega,x) \in \Lambda_0, \exists 1 \leq i < d \mbox{ such that }\lambda^{(i)}(x) < 0 \leq \lambda^{(i+1)}(x) \mbox{ and } \\ 
& \dim(V^{(1)}(x) \oplus \cdots \oplus V^{(i)}(x)) = k\}. \label{term 101}
\end{eqnarray}
Then for any $\varepsilon > 0$ there exists a closed set $\cal{K} \subset \Gamma_k$, $\varrho > 0$ such that $\mu(\Gamma_k \setminus \cal{K}) < \varepsilon$, and  $\cWs_{\omega}(x, \varrho)$ is a $k-$dimensional submanifold for any $x \in \cal{K}$, and depends continuously on $x \in \cal{K}$. 

Let $\cal{V}_1,\cal{V}_2$ be two open submanifolds of dimension $d-k$ which are uniformly transverse to the lamination $\{ \cWs_{\omega}(x, \varrho) \}_{x \in \cal{K}}$. Assume that $\cal{V}_1, \cal{V}_2$ are small enough so that $\cWs_{\omega}(x, \varrho) \bigcap \cal{V}_i$ contains at most one point for any $i=1,2$, $x\in \cal{K}$.  Let 
\begin{align*}
\cal{U} = \{x \in \cal{V}_1 | \cWs_{\omega, \varrho}(x) \mbox{  transversally intersects $\cal{V}_2$ at a unique point} \}.
\end{align*}
When $\cal{V}_1, \cal{V}_2$ are sufficiently close, $\cal{U}$ is nonempty. In this case, the holonomy map along the stable leaves $p : \cal{U} \to \cal{V}_2$ is defined as 
\aryst
 p(x) = \mbox{ the unique intersection between $\cWs_{\omega}(x, \varrho)$ and $\cal{V}_2$}, \forall x \in \cal{U}.
 \earyst
Any such $p$ is absolutely continuous in the sense that for any $A \subset \cal{U}$ with $Vol_{\cal{U}}(A) = 0$, we have $Vol_{\cal{V}_2}(p(A)) = 0$. Here $Vol_{\cal{U}}$ and $Vol_{\cal{V}_2}$ are respectively the volumes on $\cal{U}$ and $\cal{V}_2$.

\subsection{Dominated splitting and pinching condition}

\begin{defi}\label{def gras}
For any integer $1 \leq \ell \leq d-1$, we denote by $Gr(M, \ell)$ the \textit{$\ell-$subspace Grassmannian bundle over $M$}. For each $x \in M$, the fibre $Gr(M,\ell)_{x}$ is isomorphic to $Gr(T_xM, \ell)$, the Grassmannian of $\ell-$subspaces of $T_xM$.

 We denote $p_{\ell} : Gr(M, \ell) \to M$ the canonical projection.  An element of $Gr(M, \ell)$ will be denoted by $(x,E)$ for some $x \in M$ and $E \in Gr(T_xM, \ell)$. We fix an arbitrary smooth Riemannian metric on $Gr(M,\ell)$.
 
For any $(x,E) \in Gr(M, \ell)$, the tangent space of $Gr(M, \ell)$ at $(x,E)$ can be identified with
\begin{eqnarray*}
T_{(x,E)}Gr(M, \ell) = \{(v,\phi) | v\in T_xM, \phi \in \cal L(E, T_xM \big / E)\}.
\end{eqnarray*}

Let $f : M \to M$ be a $C^r$ diffeomorphism for some $r \geq 1$. We denote by  $\GG(f) : Gr(M,\ell) \to Gr(M,\ell)$ the lift of $f: M \to M$ to the fiber bundle $Gr(M,\ell)$, where for any $(x,E) \in Gr(M, \ell)$,
\begin{eqnarray*}
\GG(f)(x, E) = (f(x), Df(x,E)).
\end{eqnarray*}
It is direct to see that $\GG(f)$ is $C^{r-1}$.
\end{defi}

\begin{defi} \label{lift vector field}
Let $r \geq 1$ be fixed, as well as an integer $1 \leq \ell  \leq d-1$ and a $C^{r}$ vector field $V$ on $M$. Let $\Psi_{V} : M \times \R \to M$ be the flow generated by $V$, i.e. $\partial_{t}\Psi_{V}(x,t) = V(\Psi_{V}(x,t))$ for all $(x,t) \in M \times \R$. Define $\GG(\Psi_{V})$ by
\begin{eqnarray*}
\GG(\Psi_{V}) : Gr(M,\ell) \times \R &\to& Gr(M,\ell) \\
(x,E,t) &\mapsto& \GG(\Psi_{V}(\cdot, t ))(x,E).
\end{eqnarray*}
Then $\GG(\Psi_{V})$ is a flow generated by a $C^{r-1}$ vector field on $Gr(M,\ell)$, denoted by $\GG(V)$. Here  $\GG(V)$ is defined as follows. For each $(x,E) \in Gr(M,\ell)$, 
\begin{eqnarray} \label{eq 90}
\GG(V)(x,E) = (V(x), \phi ),
\end{eqnarray}
where $\phi \in \cL(E, T_xM \big / E)$ is given by
\begin{eqnarray}\label{eq 100}
\phi(u) = DV(x, u) + E, \quad \forall u \in E.
\end{eqnarray}
We will briefly denote $\GG(V)(x,E)$ by
\begin{eqnarray} \label{eq 9022}
\GG(V)(x,E) = (V(x), DV(x,\cdot|E) ). 
\end{eqnarray}
\end{defi}

\begin{defi} \label{def partially hyperbolic}

A $C^{1}$ diffeomorphism $f : M \to M$ has a non-trivial \textit{uniformly dominated splitting} if the following conditions hold. There is a nontrivial continuous splitting of the tangent bundle $TM = E_1 \oplus  E_{2}$ that is invariant under $Df$. Furthermore, there exist constants $\bar{\chi}_1 <  \hat{\chi}_1$ such that 
\begin{eqnarray}
 \norm{Df(x,v)} &<& e^{\bar{\chi}_1} \norm{v},   \forall x \in M,  v \in E_{1}(x)\setminus \{0\}  \label{ds ineq 3}, \\
 \norm{Df(x,u)}  &>& e^{\hat{\chi}_1}\norm{u},  \forall x \in M,  u \in E_{2}(x)\setminus \{0\} \label{ds ineq 32}. 
\end{eqnarray} 

\end{defi}

The above definition includes any partially hyperbolic diffeomorphism, defined as follows.
\begin{defi} \label{def ph}

A $C^{1}$ diffeomorphism $f : M \to M$ is \textit{partially hyperbolic} if the following conditions hold. There is a nontrivial continuous splitting of the tangent bundle $TM = \Es \oplus \Ec \oplus \Eu$, that is invariant under $Df$. Furthermore, there exists $\bar{\chi}^{u},  \bar{\chi}^{s} > 0$ and $\bar{\chi}^{c}, \hat{\chi}^{c} \in \R$  such that
\begin{eqnarray} \label{ph ineq 1}
 -\bar{\chi}^{s} < \bar{\chi}^{c }\leq \hat{\chi}^{c} < \bar{\chi}^{u},
\end{eqnarray}
and we have

\begin{eqnarray} 
&\norm{Df(x, v)}& < e^{-\bar{\chi}^{s}}\norm{v} ,\forall x \in M, v \in \Es(x) \setminus \{0\} \label{ph ineq 2},\\
e^{\bar{\chi}^c}\norm{v} < &\norm{Df(x, v)}& < e^{\hat{\chi}^{c}}\norm{v} ,\forall x \in M,  v \in \Ec(x) \setminus \{0\} \label{ph ineq 3},\\
 e^{\bar{\chi}^{u}}\norm{v} < &\norm{Df(x, v)}& ,\forall x \in M,  v \in \Eu(x) \setminus \{0\}  \label{ph ineq 4}.
\end{eqnarray}

\end{defi}

It is clear that any partially hyperbolic diffeomorphisms has a non-trivial uniformly dominated splitting given by $TM = E^s \oplus (E^c \oplus E^u)$ ( or $TM = E^s \oplus (E^c \oplus E^u)$ ).

The lower bound for the H\"older exponents of $E_1$,$E_2$ in Theorem \ref{thm transitive group} will be used in an essential way. In fact, for any $C^2$ diffeomorphism $f$ with a dominated splitting $TM = E_1 \oplus E_2$, both $E_1(x), E_2(x)$ are H\"older continuous in $x$. This follows from H\"older section theorem (3.8) in \cite{HPS}. 

 The following notion of pinching for diffeomorphisms with a dominated splitting allows us to verify numerically the needed H\"older regularity in Theorem \ref{thm transitive group}. 
\begin{defi}[ Pinching for uniformly dominated splittings ]  \label{def pin ds} 
Let $\theta \in (0,1)$. A diffeomorphisms $f : M \to M$ having a non-trivial uniformly dominated splitting $TM = E_{1} \oplus E_{2}$ is said to be \textit{$\theta-$pinched} if the following holds. There exist constants  $(\bar{\chi}_1, \hat{\chi}_1)$ satisfying  \eqref{ds ineq 3}, \eqref{ds ineq 32} and $\hat{\chi}^u, \hat{\chi}^s > 0$ such that
\begin{eqnarray}
e^{-\hat{\chi}^s} \norm{v} < &\norm{Df(x,v)}&< e^{\hat{\chi}^u} \norm{v}  ,\forall x \in M, v\in T_xM \setminus \{0\} \label{theta pinching ds ineq 02},
\end{eqnarray}
\begin{eqnarray}\label{theta pinching ds ineq 1}
\max(\hat{\chi}^{u},\hat{\chi}^{s})  \theta <  \hat{\chi}_1 - \bar{\chi}_1.
\end{eqnarray} 
\end{defi}

The pinching condition is commonly used  in the study of partially hyperbolic systems. For example, Theorem A in \cite{PSW} says that  $\theta-$pinching implies that the $s,u-$holonomy maps for a $C^2$ partially hyperbolic diffeomorphism are $\theta-$H\"older. What we need from the pinching condition for dominated splitting is the following much weaker statement that can be proved via the standard H\"older section theorem in \cite{HPS}. We omit the proof of the following proposition for it is well-known, see for example \cite{HPS}, chapter 3, (3.8).

\begin{prop} \label{prop theta pin theta holder}
Let  $r \geq 2$ \clb, $\theta \in (0,1)$. Let $f : M \to M$ be a $C^{r}$ diffeomorphism having a $\theta-$pinched uniformly dominated splitting $TM =E_{1} \oplus E_{2}$ with constants  $\bar{\chi}_1, \hat{\chi}_1, \hat{\chi}^s, \hat{\chi}^u$ satisfying \eqref{ds ineq 3}, \eqref{ds ineq 32}, \eqref{theta pinching ds ineq 02} and \eqref{theta pinching ds ineq 1}, then $E_{1}(x), E_{2}(x)$ are $\theta-$H\"older functions of $x$.
\end{prop}
The following is an immediate consequence of Proposition \ref{prop theta pin theta holder}.
\begin{cor}\label{cor pinch}
Theorem \ref{thm transitive group} applies to any $\frac{d}{d+1}-$pinched $C^r$ volume preserving partially hyperbolic diffeomorphism.
\end{cor}
\begin{proof}
Any $\frac{d}{d+1}-$pinched partially hyperbolic diffeomorphism is also $\theta-$pinched for some $\theta \in (\frac{d}{d+1},1)$ by Definition \ref{def pin ds}. Then our corollary follows immediately from Proposition \ref{prop theta pin theta holder}.
\end{proof}

\section{From averaged uniform growth to metric transitivity}\label{main section}

From now on, we let $(M, d_M)$ be a smooth $d-$dimensional compact Riemannian manifold endowed with the volume form $\mu$ given by the metric.

\subsection{Averaged uniform growth}\label{Averaged uniform growth}
To simplify the notations, we introduce the following quantities.
\begin{nota}
Given an IFS $\{f_i\}_{i=1}^{m} \subset \textnormal{\Diff}^1(M)$, an integer $b\in [1,d-1]$, for any $x \in M$, any $(d-b)-$subspace  $E \subset T_{x}M$, any $n \in \N$, we denote
\begin{eqnarray*}
C(x,E,n) &=&  \EV(\log \sup_{v \in U(E^{\perp})} \norm{P_{{Df_{\omega}^{n}(x,E)^{\perp}}}(Df_{\omega}^{n}(x,v) )} ), \\
D(x,E,n) &=&  \EV( \log \inf_{u \in U(E)} \norm{Df_{\omega}^{n}(x,u) } ).
\end{eqnarray*}
Here for any linear subspace of $T_{x}M$ denoted by $G$, we define $U(G)$ to be the set of unit vectors in $G$; for any $v \in T_{x}M$, we denote by $P_{G}(v)$ the orthogonal projection of $v$ to $G$. 
\end{nota}

In loose terms, $C(x,E,n)$ describes the weakest contraction rate in the directions transverse to $E$ for a typical iteration of length $n$ starting from $x$; $D(x,E,n)$ describes the strongest contraction ( or the weakest expansion ) in $E$  for a typical iteration of length $n$ starting from $x$.

\begin{defi}\label{def what is uniform}
Given any IFS $\{f_i\}_{i=1}^{m} \subset \textnormal{\Diff}^1(M)$, an integer $b \in [1,d-1]$,
if there exists an integer $n_0 \geq 1$, real numbers $\kappa_1 > 0$, $\kappa_2 \in (-\infty, \kappa_1) $ such that for any $x \in M$, any $(d-b)-$subspace $E \subset T_{x}M$, we have
\begin{eqnarray}
\frac{1}{n_0}C(x,E,n_0) &<& -\kappa_1  \label{uni ineq 1}, \\
\frac{1}{n_0}D(x,E,n_0) &>& -\kappa_2  \label{uni ineq 2}.
\end{eqnarray}
Then we say that $\{f_i\}_{i=1}^{m}$ is \textit{$(n_0, \kappa_1, \kappa_2, b)-$uniform}.

If there exists $b \in [1,d-1]$ such that $\{f_i\}_{i=1}^{ m}$ is $(n_0, \kappa_1, \kappa_2, b)-$uniform then we say that $\{f_i\}_{i=1}^{m}$ is \textit{$(n_0, \kappa_1, \kappa_2)-$uniform}. By definition, it is clear that being $(n_0, \kappa_1,\kappa_2)-$uniform is a $C^1$ open property.

\end{defi}

The following is the main result of this section.
\begin{prop}\label{prop main criterion}
Let $\epsilon > 0$ be fixed, as well as an integer $m \geq 2$. Let $\{ f_i \}_{i=1}^{m} \subset \diff^{1+\epsilon}(M, \mu)$ be an IFS that is $(n_0, \kappa_1, \kappa_2)-$uniform for some integer $n_0 \geq 1$, real numbers $\kappa_1 > 0$, $\kappa_2 \in (-\infty, \kappa_1)$. Then any $\{ f_i \}_{i=1}^{m}-$invariant closed set $\cal{K}$ strictly contained in $M$ satisfies $\mu(\cal{K}) = 0$. In particular, $\{ f_i \}_{i=1}^m$ is transitive. 
\end{prop}

The proof of Proposition \ref{prop main criterion} is divided into two parts, contained in the next two Subsections. We will give the proof of Proposition  \ref{prop main criterion} at the end of Subsection \ref{A criteria for metric transitivity}.

\subsection{Stable manifolds with uniformly lower bounded size}\label{Stable manifolds with uniformly lower bounded size}

\begin{lemma}\label{lemma ldt}
For any integers $n_0 \geq 1$, $b \in [1,d-1]$, real numbers $\kappa_1 > 0$, $\kappa_2 \in (-\infty, \kappa_1)$,
for any IFS $\{f_i\}_{i=1}^{m}$ that is $(n_0,\kappa_1,\kappa_2,b)-$uniform,  there exist $C_1, \sigma > 0$, such that the following is true.
For any $x \in M$, any $(d-b)-$subspace $E \subset T_{x}M$, any integer $n \geq 0$, we have
\begin{align*}
\EV(\sup_{v \in U(E^{\perp})}{\norm{P_{Df^{n}_{\omega}(x,E)^{\perp}}(Df^{n}_{\omega}(x,v))}}^{\sigma}) < C_1 e^{-n \sigma \kappa_1},
\end{align*}
\begin{align*}
\EV((\inf_{v \in U(E)}\norm{{Df^{n}_{\omega}(x,v)}})^{-\sigma}) < C_1 e^{n \sigma \kappa_2}. 
\end{align*}
\end{lemma}
\begin{proof}
The proof is similar to that of Lemma 4 in \cite{DK}.

By our hypothesis,  we have 
\begin{eqnarray*}
\sup_{\substack{x \in M \\ E \in Gr(T_{x}M, d-b) }}\EV(\log \sup_{v \in U(E^{\perp})} \norm{P_{Df^{n_0}_{\omega}(x,E)^{\perp}}(Df^{n_0}_{\omega}(x,v))} ) < - n_0 \kappa_1.
\end{eqnarray*}
Then by $e^{x} = 1+ x + O(x^{2})$ when $x \in (-1,1)$, for small $\sigma > 0$  ( depending on $n_0$ and $\norm{f_i}_{C^{1}}$, $i=1,\cdots, m$ ), we have
\begin{eqnarray*}
&&\sup_{\substack{x \in M \\ E \in Gr(T_{x}M, d-b) }}\EV(\sup_{v \in U(E^{\perp})} \norm{P_{Df^{n_0}_{\omega}(x,E)^{\perp}}(Df^{n_0}_{\omega}(x,v))}^{\sigma} )  \\
&=&1 + \sigma \sup_{\substack{x \in M \\ E \in Gr(T_{x}M, d-b) }}\EV(\log \sup_{v \in U(E^{\perp})} \norm{P_{Df^{n_0}_{\omega}(x,E)^{\perp}}(Df^{n_0}_{\omega}(x,v))} ) + O(\sigma^{2}).
\end{eqnarray*}
Then by \eqref{uni ineq 1}, when $\sigma > 0$ is sufficiently small, we have
\begin{eqnarray*}
\sup_{\substack{x \in M \\ E \in Gr(T_{x}M, d-b) }}\EV(\sup_{v \in U(E^{\perp})} \norm{P_{Df^{n_0}_{\omega}(x,E)^{\perp}}(Df^{n_0}_{\omega}(x,v))}^{\sigma} )  < e^{-\sigma n_0 \kappa_1}.
\end{eqnarray*}
By submultiplicativity, for each integer $k > 0$ we have
\begin{eqnarray*}
\sup_{\substack{x \in M \\ E \in Gr(T_{x}M, d-b) }}\EV(\sup_{v \in U(E^{\perp})} \norm{P_{Df^{l n_0}_{\omega}(x,E)^{\perp}}(Df^{k n_0}_{\omega}(x,v))}^{\sigma} )  < e^{-\sigma k n_0 \kappa_1}.
\end{eqnarray*}
Then for some large constant $C_1$, for all $n \in \N$ we have
\begin{eqnarray*}
\sup_{\substack{x \in M \\ E \in Gr(T_{x}M, d-b) }}\EV(\sup_{v \in U(E^{\perp})} \norm{P_{Df^{n}_{\omega}(x,E)^{\perp}}(Df^{n}_{\omega}(x,v))}^{\sigma} )  < C_1 e^{-\sigma n \kappa_1}.
\end{eqnarray*}
This proves the first inequality. 
The second inequality follows from a similar argument.
\end{proof}

In order to state our main proposition properly, we will have to quantify transversality between linear subspaces and curves on the manifold.

\begin{defi}
Given any $x \in M$,   for any real numbers $l, \rho > 0$ , a $C^{1}$ curve $\mathcal{S} \subset M$ through $x$ is said to be \textit{ $(l, \rho)-$ good} if the following holds: 

(1)  $10(l+ l \rho)$ is smaller than the injectivity radius of $M$ and there exists a $C^{1}$ function $\psi : T_{x}\cal{S}(l) \to (T_{x}\cal{S})^{\perp}$, such that $\exp_{x}(\textnormal{graph}(\psi)) \subset \cal{S}$, where $T_{x}\cal{S}(l)$ denotes the interval of radius $l$ centered at the origin of $T_{x}\cal{S}$;

(2) $\psi(0) = 0$ and $\norm{\psi}_{C^1} < \rho$.

Given any $x \in M$, any hyperplane $E \subset T_{x}M$, for any $l, \rho, \alpha > 0$, a $C^1$ curve $\mathcal{S} \subset M$ through $x$ is said to be \textit{$(l, \rho, \alpha)-$regular} with respect to $E$ if $\mathcal{S}$ is $(l,\rho)-$good and :

(3) The angle between $T_{x}\mathcal{S}$ and $E$ is not less than $\alpha$.

We say that a subset $\cD \subset M$ containing $x$ is $(l, \rho, \alpha)-$regular with respect to $E$ if there exists a $C^1$ curve $\cal{S} \subset \cD$ that contains $x$, and is $(l, \rho, \theta)-$regular with respect to $E$.

\end{defi}

The following proposition gives a useful property of a uniform IFS.

\begin{prop} \label{prop uniform to uniform bound}
Let $\epsilon > 0$ be fixed, as well as an integer $m \geq 2$. Let $\{ f_i \}_{i=1}^{m} \subset \diff^{1+\epsilon}(M, \mu)$ be an IFS that is $(n_0, \kappa_1, \kappa_2)-$uniform for some integer $n_0 \geq 1$, real numbers $\kappa_1 > 0$, $\kappa_2 \in (-\infty, \kappa_1)$. 
Then there exist $l, \rho, \theta > 0$ such that for all $x$ in a co-null subset of $M$, for any hyperplane $E \subset T_{x}M$, we have
\begin{eqnarray*}
\PP(\mbox{$\cWs_{\omega}(x)$ is $(l, \rho, \theta)-$ regular with respect to $E$}) > \frac{1}{2}.
\end{eqnarray*}
\end{prop}
\begin{proof} 
By definition, $\{f_i\}_{i=1}^{m}$ is $(n_0, \kappa_1, \kappa_2, b)-$uniform for some  integer $b \in [1,d]$. Let $\Omega = \{1,\cdots, m\}^{\Z}$ and let $\Lambda_0 \subset \Omega \times M$ be given by Theorem \ref{LE random dyn}.

By applying Lemma \ref{lemma ldt} and Borel-Cantelli lemma, we see that: for all $(\omega, x)$ in a co-null set in $\Lambda_0$, the Lyapunov exponents are not all zero. Without loss of generality, we can assume that $\cWs_{\omega}(x)$ are $C^{1}-$manifolds for all $(\omega,x) \in \Lambda_0$.

We claim that we can give a lower bound for the regularity of these stable manifolds uniformly in $(x,E) \in Gr(M, d-b)$ for $x$ in a co-null set. For this purpose, we first establish some \textit{a priori} estimates.

\begin{lemma}\label{a priori est 3}
For any $(\omega,x) \in \Lambda_0$, set $\gamma(x), \{ V^{(j)}(\omega, x)\}_{j=1}^{\gamma(x)}$ as in Theorem \ref{LE random dyn} , let $q \in [1,\gamma(x)]$ be the smallest integer such that $\dim \sum_{j=1}^{q} V^{(j)}(\omega,x)$ is at least $b$, and let $V^{-}(\omega, x) =\sum_{j=1}^{q} V^{(j)}(\omega,x)$. Then for every $x$ in a co-null set of $M$,  for each subspace $E \subset T_{x}M$ of dimension $(d-b)$, we have
\begin{eqnarray*}
\PP( V^{-}(\omega,x) \bigcap E \neq \{0\} ) = 0.
\end{eqnarray*}
Moreover, $\dim(V^{-}(\omega,x)) = b$ for $(\PP \times \mu)-a.e. (\omega,x)$.
\end{lemma}
\begin{proof}
Assume to the contrary that there exists a positive measure set $M_0 \subset M$ such that for every $x \in M_0$, $\PP-a.e. \omega$ satisfy that $(\omega,x) \in \Lambda_0$, and there exists a subspace $E \subset T_{x}M$ of dimension $(d-b)$ such that $\PP(V^{-}(\omega,x) \bigcap E \neq \{0\}) > 0$. 

We fix an arbitrary $x \in M_0$.
For $\PP-a.e. \omega$ such that $(\omega,x) \in \Lambda_0$ and $V^{-}(\omega,x) \bigcap E \neq \{0\}$, for any $v \in V^{-}(\omega,x) \bigcap E$ we have 
\begin{align}\label{a priori est 3 ineq 1}
\lim_{n \to \infty}\frac{1}{n}\log \norm{Df^{n}_{\omega}(x, v)} \leq \lambda^{(q)}(x).
\end{align}
By the choice of $q$, let $V' := \sum_{j=1}^{q-1}V^{(j)}(\omega,x)$, we have $\dim V' < b$. Then $ V' + E \subsetneq T_{x}M$. 
 It is easy to choose a basis of $V' + E$, denoted by $\xi_1, \cdots, \xi_{\ell}$ for some integer $1 \leq \ell < d$, and vectors $\xi_{\ell +1}, \cdots, \xi_{d}$, such that $\xi_1, \cdots, \xi_d$ form a basis of $T_{x}M$, and satisfies that for each $1 \leq i \leq \gamma(x)$, the number of indexes $j$ satisfying \eqref{basis le} equals to $\dim V_i(\omega,x)$.
Let $u = \xi_{r+1}$. By Theorem \ref{LE random dyn} (2), we have 
\begin{eqnarray}\label{anglenottoosmall}
\lim_{n \to \infty} \frac{1}{n}\log |\angle( Df^{n}_{\omega}(x,u), Df^{n}_{\omega}(x,E))| = 0.
\end{eqnarray}
By $u \notin V'$, we have
 \begin{align}\label{a priori est 3 ineq 2}
 \lim_{n \to \infty}\frac{1}{n}\log \norm{P_{Df^{n}_{\omega}(x,E)^{\perp}}(Df^{n}_{\omega}(x,u))} \geq \lambda^{(q)}(x).
 \end{align}
Denote by $\Omega_0$ the set of events $\omega$  simultaneously realising \eqref{a priori est 3 ineq 1} for some $v \in E$, and \eqref{a priori est 3 ineq 2} for some $u \notin E$. Then we have
\ary \label{term 2000}
\PP(V^{-}(\omega, x) \bigcap E \neq \{0\}) \leq \PP(\Omega_0).
\eary
Moreover for any $\omega \in \Omega_0$, we have
\begin{eqnarray*}
\liminf_{n \to \infty}( \sup_{u \in U(E^{\perp})} \frac{1}{n} \log \norm{P_{Df^{n}_{\omega}(x,E)^{\perp}}(Df^{n}_{\omega}(x,u))} - \inf_{v \in U(E)} \frac{1}{n}\log \norm{Df^{n}_{\omega}(x, v)}) \geq 0.
\end{eqnarray*}

\textbf{Claim.} We have $\PP(\Omega_0) = 0$.
\begin{proof}
Assume to the contrary that $\PP(\Omega_0) > 0$. Let $\kappa' = \frac{\kappa_1 - \kappa_2}{4}$.
Then for all sufficiently large $n$, the following event of $\omega$ has probability at least  $\frac{1}{2}\PP(\Omega_0)$ :
\aryst
 \sup_{u \in U(E^{\perp})} \frac{1}{n} \log \norm{P_{Df^{n}_{\omega}(x,E)^{\perp}}(Df^{n}_{\omega}(x,u))} - \inf_{v \in U(E)} \frac{1}{n}\log \norm{Df^{n}_{\omega}(x, v)} > -\kappa'.
 \earyst
Combined with Lemma \ref{lemma ldt} and Cauchy's inequality, we obtain for all sufficiently large $n$ that 
 \begin{eqnarray*}
 \frac{1}{2}\PP(\Omega_0)e^{-\frac{1}{2}n \sigma \kappa'} &<&
\EV((\sup_{v \in U(E^{\perp})}{\norm{P_{Df^{n}_{\omega}(x,E)^{\perp}}(Df^{n}_{\omega}(x,v))}})^{\frac{\sigma}{2}} (\inf_{v \in U(E)}\norm{{Df^{n}_{\omega}(x,v)}})^{-\frac{\sigma}{2}}) \\
 &<& \EV(\sup_{v \in U(E^{\perp})}{\norm{P_{Df^{n}_{\omega}(x,E)^{\perp}}(Df^{n}_{\omega}(x,v))}}^{\sigma})^{\frac{1}{2}}  \EV((\inf_{v \in U(E)}\norm{{Df^{n}_{\omega}(x,v)}})^{-\sigma})^{\frac{1}{2}} \\
 &<&C_1^2e^{-n\sigma \frac{\kappa_1 - \kappa_2}{2}}
\end{eqnarray*}
By letting $n$ tend to infinity, we obtain a contradiction. This proves our claim.
\end{proof}
We obtain a contradiction by the above claim and \eqref{term 2000}. This concludes the proof.
\end{proof}

By Lemma \ref{a priori est 3}, we see that for $x$ in a co-null set in $M$, for any $E \subset T_{x}M$ of dimension $(d-b)$, we have $\dim V^{-}(\omega,x)=b$ and $V^{-}(\omega,x) \bigcap E = \{0\}$ for $\PP-a.e. \omega$.

Now we show that with large probability, the stable manifolds at $x$ form good angles with any given $(d-b)-$dimensional subspace.
\begin{lemma} \label{good angles}
There exist constants $C_0, \beta$ such that for all $x$ in a co-null set in $M$, for any $E \subset T_{x}M$ of dimension $(d-b)$, for any $\epsilon > 0$, we have
\begin{eqnarray*}
\PP(\angle(E, V^{-}(\omega,x)) \leq \epsilon ) \leq C_0 \epsilon^{\beta}
\end{eqnarray*}
\end{lemma}
\begin{proof}
The proof of this lemma follows exactly that of Corollary 4(b) in \cite{DK}. The push-forward of any graph of linear map from $E$ to $E^{\perp}$ under $Df^{n}_{\omega}$ gets exponentially close to that of $E$ with large probability. By combining this with Lemma \ref{a priori est 3}, this shows that $E$ forms small angle with $V^{-}(\omega,x)$ with small probability. We refer the readers to \cite{DK} for details. 
\end{proof}

We set $V^{+}(\omega,x) = \oplus_{j=q+1}^{\gamma(x)} V^{(j)}(\omega,x)$, where $q \in [1,\gamma(x)]$ is given by Lemma \ref{a priori est 3}.
By definition, it is clear that $T_{x}M$ is the direct sum of $V^{-}(\omega,x)$ and $V^{+}(\omega,x)$. It is a standard fact that $V^{-}(\omega,x)$ depends only $\omega_{0}, \omega_{1},\cdots$ and $x$;  $V^{+}(\omega,x)$ depends only $\omega_{-1}, \omega_{-2},\cdots$ and $x$.

Now we define \lq\lq Pesin events \rq\rq as follows. Take any constants  $\bar{\kappa}_1 , \bar{\kappa}_2 $ such that $\bar{\kappa}_1 > 0$ and $\kappa_2 < \bar{\kappa}_2 < \bar{\kappa}_1 < \kappa_1$. We set $\epsilon_0 = \frac{1}{10} \min(\bar{\kappa}_2 - \kappa_2 , \kappa_1 - \bar{\kappa}_1)$.
For any constants $C, \epsilon > 0$, set
\begin{eqnarray*}
\Delta_{C, \epsilon}(x) = \{\omega | \norm{Df^{k, k+j}_{\omega} | V^{-} } \leq  Ce^{\epsilon k -  j \bar{\kappa}_1}, \norm{Df^{k,k-j}_{\omega}| V^{+}} \leq Ce^{\epsilon k + j \bar{\kappa}_2},  \\ \angle(V^{-}(T^{k}(\omega,x)), V^{+}(T^{k}(\omega,x))) \geq C^{-1} e^{-k \epsilon} \}.
\end{eqnarray*}

We can give lower bounds for the probabilities of Pesin events. The following is the analogue to Corollary 4(a),(c) of \cite{DK}.
\begin{lemma} \label{lemma pesin set large}
For each $\sigma > 0$, there exist $C, \epsilon > 0$ such that for any $x$ in a co-null set of $M$, we have 
\begin{eqnarray*}
\PP(\Delta_{C,\epsilon}(x)) > 1- \sigma.
\end{eqnarray*}
\end{lemma}
\begin{proof}
By Lemma \ref{good angles} and the fact that $V^{-}(T^{k}(\omega,x))$ depends only on $\omega_{k}, \omega_{k+1}, \cdots$ and $T^{k}(\omega,x)$ ; $V^{+}(T^{k}(\omega,x))$ depends only $\omega_{k-1}, \omega_{k-2}, \cdots$  and $T^{k}(\omega,x)$, we have 
\begin{eqnarray*}
\PP(\angle(V^{-}(T^{k}(\omega,x)), V^{+}(T^{k}(\omega,x)))\leq C^{-1}e^{-k\epsilon} ) \leq C_0C^{-\beta} e^{-k\beta\epsilon}.
\end{eqnarray*}
Summing up over $k \geq 1$, we see that the probability of the events that violate the last condition in the definition of $\Delta_{C,\epsilon}$ can be made arbitrarily small by making $C$ large.

For each $k$ and $j \geq 1$, for each event $\omega$ that violates the first condition for $\Delta_{C,\epsilon}$, i.e. $\norm{Df^{k,k+j}_{\omega}|V^{-}} > Ce^{\epsilon k - j \bar{\kappa}_1}$, we claim that, conditioning on $\omega_0, \cdots, \omega_{k-1}$, for any subspace $E \subset T_{f^{k}_{\omega}(x)}M$ of dimension $(d-b)$,

(1) either we have  $|\angle(V^{-}(T^{k+j}(\omega,x)), Df^{k,k+j}_{\omega}(x,E))| < C^{-\frac{1}{2}}e^{-(j+k)\epsilon_0}$; 

(2) or there exists $v \in U(E^{\perp})$ such that $\norm{P_{Df^{k,k+j}_{\omega}(x,E)^{\perp}}(Df^{k,k+j}_{\omega}(x,v))} > \frac{1}{10} C^{\frac{1}{2}}e^{\epsilon k - (k+j)\epsilon_0 - j \bar{\kappa}_1}$.  

Indeed, assume the first case fails, that is, $| \angle(V^{-}(T^{k+j}(\omega,x)), Df^{k,k+j}_{\omega}(x,E))| \geq C^{-\frac{1}{2}} e^{-(j+k)\epsilon_0}$.
We choose a unit vector $u \in V^{-}(T^{k}(\omega, x))$ such that $\norm{Df_{\omega}^{k,k+j}(x, u)} \geq Ce^{\epsilon k - j \bar{\kappa}_1}$. Then 
\begin{eqnarray*}
 &&\norm{ P_{Df^{k,k+j}_{\omega}(x,E)^{\perp}}(Df^{k,k+j}_{\omega}(x,P_{E^{\perp}}(u)))} \\
&=&\norm{P_{Df^{k,k+j}_{\omega}(x,E)^{\perp}}(Df^{k,k+j}_{\omega}(x,u)) } \\
&\geq& Ce^{\epsilon k - j \bar{\kappa}_1} |\sin \angle(V^{-}(T^{k+j}(\omega,x)), Df^{k,k+j}_{\omega}(x,E))| \\
&\geq& \frac{2}{\pi}C e^{\epsilon k  - j \bar{\kappa}_1}| \angle(V^{-}(T^{k+j}(\omega,x)), Df^{k,k+j}_{\omega}(x,E))| >  \frac{1}{10}C^{\frac{1}{2}} e^{\epsilon k - (k+j)\epsilon_0 - j \bar{\kappa}_1}.
\end{eqnarray*}
The second case is then verified for $v = \frac{P_{E^{\perp}}(u)}{\norm{P_{E^{\perp}}(u)}}$.

 Choose $\epsilon = 2\epsilon_0$ and $j \geq n_0$. By Lemma \ref{lemma ldt}, the probabilities of the events in (1),(2) are bounded by $C_0C^{-\frac{\beta}{2}}e^{-(j+k)\beta \epsilon_0}$ and $10C_1C^{-\frac{\sigma}{2}}e^{-\sigma (k+j)\epsilon_0}$ respectively. Since $\{e^{- \sigma ( j+k) \epsilon_0} + e^{- (j+k) \beta \epsilon_0}\}_{k \geq 1, j\geq 1}$ are summable, we see that the probabilities of the events that violet the first condition for $\Delta_{C,\epsilon}$ can be made arbitrarily small by making $C$ large.

For the second condition, we note that for each $k \in \N$, $1\leq j \leq k$, $V^{+}(T^{k-j}(\omega,x))$ depends only on $\omega_{k-j-1}, \omega_{k-j-2}, \cdots$. For each $\omega$ violating the second condition, i.e. $\norm{Df^{k,k-j}_{\omega}|V^{+}} > C e^{\epsilon k + j \bar{\kappa}_2}$, we have $\inf_{v \in V^{+}(T^{k-j}(\omega,x))}\norm{Df^{k-j,k}_{\omega}(v)} < C^{-1}e^{- \epsilon k - j \bar{\kappa}_2 }$.
By Lemma \ref{lemma ldt} for $E = V^{+}(T^{k-j}(\omega,x))$, we see that the probabilities of the events that violate the second condition for $k,j$ are bounded by $C_2C^{-\sigma}e^{-\sigma(j+k)\epsilon}$. Then by the same reasoning, we can make the probability of violating the second condition arbitrarily small by making $C$ large.
This proves the lemma. 

\end{proof}

Now by the construction of stable manifolds in the theory of non-uniformly partially hyperbolic systems, we get the following lemma.
\begin{lemma} \label{lemma pesin set of uniform bound}
For each $C, \varepsilon > 0$, there exist constants $l, \rho > 0$ such that for any $x$ in a co-null set of $M$, any $\omega \in \Delta_{C,\varepsilon}(x)$, 
$\cWs_{\omega}(x)$ is $(l, \rho)-$good.

\end{lemma}
\begin{proof}
This is a consequence of Theorem 2.1.1 in \cite{P}.
\end{proof}

Now Proposition \ref{prop uniform to uniform bound} follows from Lemma \ref{good angles}, \ref{lemma pesin set large} and \ref{lemma pesin set of uniform bound}.

\end{proof}

\subsection{A criterion for metric transitivity} \label{A criteria for metric transitivity}

The following proposition is at the core of this section.

\begin{prop} \label{main prop}
Let  $\epsilon > 0$ be fixed, as well as an integer $m \geq 2$. Let $\{f_i\}_{i=1}^{m}$  be an IFS  consisted of diffeomorphisms in $\diff^{1+\epsilon}(M,\mu)$. If there exist $l, \rho, \theta > 0$ such that for any $x$ in a co-null subset of $M$, for any hyperplane $E \subset T_{x}M$, we have
\begin{eqnarray} \label{condition un bounded st manfld}
\PP(\mbox{$\cWs_{\omega}(x)$  is $(l, \rho, \theta)-$ regular with respect to $E$}) > \frac{1}{2}.
\end{eqnarray}
Then any invariant closed set for the IFS $\{f_i\}_{i=1}^{m}$ having positive volume equals to $M$. In particular, $\{f_i\}_{i=1}^{m}$ is transitive.
\end{prop}
\begin{proof}
Assume to the contrary that there exists a closed set $\Gamma \subsetneqq M$ that is invariant under $f_i$ for all $1 \leq i \leq m$, and $\mu(\Gamma) > 0$. Denote $\Gamma'$ the set of Lebesgue density point of $\Gamma$. Since $f_i$ is nonsingular with respect to $\mu$ for all $1\leq i\leq m$, we have $f_i(\Gamma') = \Gamma'$ for all $1\leq i\leq m$. Thus replacing $\Gamma$ by $\overline{\Gamma'}$, we can assume that for any $x \in \Gamma$, any open set $U$ containing $x$, we have $\mu(\Gamma \bigcap U) > 0$.

Take an arbitrary $y \in M \setminus \Gamma$ that is within distance $\varrho_{in}$ to $\Gamma$ ( here $\varrho_{in}$ is the injectivity radius of $M$ ), there exists $x \in \Gamma$ such that $d_{M}(x,y) = \inf_{z \in \Gamma} d_{M}(z,y)$. We define the sphere
\begin{eqnarray*}
Z = \{x' \in M | d_{M}(x', y) = d_{M}(x,y)\}.
\end{eqnarray*}
Let $E \subset T_{x}M$ be the hyperplane tangent to $Z$ at $x$. Then there exists an open neighbourhood of $x$, denoted by $U$, such that : for any $z \in U$, any curve $\cC$ through $z$ that is $(l, \rho, \theta)-$ regular with respect to the hyperplane $E' \subset T_{z}M$, where $E'$ is obtained from $E$ through a local parallel translation, we have $\cC \bigcap \Gamma^{c} \neq \emptyset$ ( Since some point in $\cC$ has distance to $y$ smaller than $d_{M}(x,y)$).

Since $f_i$ preserves $\mu$ for all $1\leq i \leq m$,  we know that for $\PP-a.e.$ $ \omega$, the local stable manifold $\cWs_{\omega,loc}(x)$ is defined for $\mu-a.e.$ $x$, and depends measurably on $x$. By Lusin's theorem and the absolute continuity of the stable lamination, we can find a set $\Omega_0 \subset \Omega$ such that

1. $\PP(\Omega_0) > \frac{9}{10}$;

2. For each $\omega \in \Omega_0$, there exists  $\Gamma_{\omega} \subset \Gamma$, such that : 

(2.I) $\mu(\Gamma \setminus \Gamma_{\omega}) < \frac{1}{10}\mu(\Gamma \bigcap U)$;

(2.II) $\cWs_{\omega,loc}(z)$ is a $C^{1}-$manifold and depends continuously on $z$ restricted to $z \in \Gamma_{\omega}$;  

(2.III) For any $z \in \Gamma_{\omega}$, any open set $V$ containing $z$, we have $\mu(\Gamma_{\omega} \bigcap V) > 0$;

(2.IV) $\Gamma_{\omega} \subset \bigcup_{k=1}^{d-1}\Gamma_k$. For any $1 \leq  k\leq d-1$, $\Gamma_{\omega} \bigcap \Gamma_k$ is a closed set and $\cWs_{\omega, loc}$ restricted to $\Gamma_{\omega} \bigcap \Gamma_{k}$ is an absolutely continous lamination. Here $\Gamma_k$ is defined in  \eqref{term 101}.

We claim that there exists $z \in \Gamma \bigcap U$ such that 
\begin{eqnarray} \label{find z}
\PP(z \in \Gamma_{\omega}) > \frac{3}{4}.
\end{eqnarray}
Indeed, we denote $\overline{\mu} = \frac{1}{\mu(\Gamma \bigcap U)} \mu|_{\Gamma \bigcap U}$ the normalised volume form restricted to $\Gamma \bigcap U$ ( this is well-defined because $\mu(\Gamma \bigcap U) > 0$), then
\ary \label{term 2030}
(\PP \times \overline{\mu})\{(\omega, z) | z \in \Gamma_{\omega}\} \geq \PP(\Omega_0) \times \frac{9}{10} > \frac{3}{4}.
\eary
By Fubini's theorem, there exists $z \in \Gamma \bigcap U$ that satisfies \eqref{find z}.

From the hypothesis of the proposition, we can assume without loss of generality that
\begin{eqnarray*}
\PP(\mbox{$\cWs_{\omega}(z)$ is $(l, \rho, \theta)-$regular with respect to $E'$}) > \frac{1}{2}.
\end{eqnarray*}
where $E' \subset T_{z}M$ is obtained from $E$ through a local parallel translation.
Then by \eqref{term 2030},
\begin{eqnarray*}
\PP(z \in \Gamma_{\omega}\bigcap U \mbox{, $\cWs_{\omega}(z)$  is $(l, \rho, \theta)-$ regular with respect to $E'$}) > \frac{1}{4}.
\end{eqnarray*}
By the choice of $U$, we see that
\begin{eqnarray*}
\PP(z \in \Gamma_{\omega}\bigcap U \mbox{,  $\cWs_{\omega}(z) \bigcap \Gamma^{c} \neq \emptyset$}) > \frac{1}{4}.
\end{eqnarray*}
By (2.IV), we can assume that $z \in \Gamma_{\omega} \bigcap \Gamma_k$ for some $1 \leq k \leq d-1$.
By (2.II), for each $\omega \in \Omega_0$ such that $z \in \Gamma_{\omega}$, there exists an open set $V_{\omega}$ containing $z$ such that: for any $z' \in V_{\omega} \bigcap \Gamma_{\omega}\bigcap \Gamma_k$, $\cWs_{\omega}(z')$ is defined and $\cWs_{\omega}(z') \bigcap \Gamma^c \neq \emptyset$.
  
We choose a local submanifold $Q$ contained in $\Gamma^c$ of dimension $d-k$ through a point in $\cWs_{\omega, loc}(z) \bigcap \Gamma^c$, uniformly transverse to the lamination $\cWs_{\omega, loc}$ on $V_{\omega} \bigcap \Gamma_{\omega} \bigcap \Gamma_k$.

By (2.II), after possibly reducing the size of $V_{\omega}$, we can assume that $\cWs_{\omega, loc}(z')$ transversally intersects $Q$ for all $z' \in V_{\omega} \bigcap \Gamma_{\omega} \bigcap \Gamma_k$.
By (2.III), we have   $\mu(V_{\omega} \bigcap \Gamma_{\omega} \bigcap \Gamma_k )>0$.

By the absolute continuity of the lamination $\cWs_{\omega, loc}$, there exists a set $Q_0 \subset Q$ of positive submanifold volume such that any $w \in Q_0$ belongs to $\cWs_{\omega, loc}(z')$ for some $z' \in V_{\omega}\bigcap\Gamma_{\omega} \bigcap \Gamma_k$.
The same is true for all submanifold $Q'$ sufficiently close to $Q$, that is, there exists a set $Q'_0 \subset Q'$  of positive submanifold volume such that any $w' \in Q'_0$ belongs to $\cWs_{\omega, loc}(z')$ for some $z' \in V_{\omega}\bigcap\Gamma_{\omega} \bigcap \Gamma_k$.

Thus there exists a set $\cal{K}_{\omega} \subset \Gamma^c$ of positive volume such that for each $w \in \cal{K}_{\omega}$, there exists $z' \in V_{\omega} \bigcap \Gamma_{\omega}$ such that $w \in \cWs_{\omega}(z')$. \clb
  Then
\begin{eqnarray*}
\PP(\mu(w \in \Gamma^c | \mbox{there exists $z' \in \Gamma$ such that $w \in \cWs_{\omega}(z')$} ) > 0) > 0.
\end{eqnarray*}
By Fubini's theorem, we have
\begin{eqnarray*}
\mu(w \in \Gamma^c | \PP(\mbox{there exists $z' \in \Gamma$ such that $w \in \cWs_{\omega}(z')$}) > 0) > 0.
\end{eqnarray*}

Recall that $\PP \times \mu$ is $T-$invariant.
By Poincar\'e reccurence theorem, for $(\PP \times \mu)-a.e.$ $(\omega, w)$, the trajectory $\{f_{\omega}^{n}(w)\}_{n \in \N}$ accumulates at $w$.
Then there exist $w \in \Gamma^c$, $z' \in \Gamma$, such that $w \in \cWs_{\omega}(z')$ and $\liminf_{n \to \infty}d(w, f^{n}_{\omega}(w)) = 0$. Then $\liminf_{n \to \infty}d(w, f^{n}_{\omega}(z')) =0$. This gives a contradiction since $f^{n}_{\omega}(z') \in \Gamma$ for all $n$ and $d_M(w, \Gamma) > 0$.
\end{proof}

We are now ready to prove Proposition \ref{prop main criterion}.
\begin{proof}[Proof of Proposition \ref{prop main criterion} :]
By combining Proposition \ref{prop uniform to uniform bound} and \ref{main prop}.
\end{proof}

\section{Transversality and growth}\label{Transversality and growth}

In this section, we show how to verify conditions in Proposition \ref{prop main criterion} under the presence of a uniformly dominated splitting. The content of this section is only used to prove Theorem \ref{thm transitive group}.

Recall that $M$ is a smooth compact Riemannian manifold of dimension $d$.
\begin{defi}
Given any integer $m \geq 2$,$1\leq \ell \leq d-1$, real number $\eta > 0$.  A continous $\ell-$subspace distribution $P$ is a continuous function from $M$ to $Gr(M,\ell)$, assigning each point $x \in M$ to $P(x) \in Gr(T_{x}M, \ell)$. We say a set of diffeomorphisms $\{f_{i}\}_{i=1}^{ m}$ is \textit{$\eta-$nontransverse to $P$} if the following holds: for any $x \in M$, any $E \in Gr(T_{x}M, d-\ell)$, we have
\begin{eqnarray*}
\#\{ i   | \mbox{ $Df_{i}(x,E)$ is transverse to $P(f_i(x))$ }\} > (1-\eta)m.
\end{eqnarray*}
\end{defi}

Proposition \ref{main prop} can be combined with the following proposition to construct transitive IFS.

\begin{prop} \label{prop unif to trans}
Given $\epsilon > 0$, integer $b \in [1, d-1]$ and  $g \in \textnormal{\Diff}^{1+\epsilon}(M, \mu)$ having a uniformly dominated splitting  $TM = E_{1} \oplus E_{2}$
with $\dim(E_{1}) = b$.
Let $\bar{\chi}_1 < 0,\hat{\chi}_1 $ satisfy \eqref{ds ineq 3} and \eqref{ds ineq 32} for $g$ in place of $f$. Take any $\hat{\chi} > \norm{f}_{C^1}$.
 Then for any \begin{eqnarray} \label{eta eta}
 \eta \in (0, \frac{1}{8}\min(\frac{\hat{\chi}_1 - \bar{\chi}_1}{ \hat{\chi} - \bar{\chi}_1}, \frac{ - \bar{\chi}_1}{ \hat{\chi} - \bar{\chi}_1},\frac{\hat{\chi}_1 - \bar{\chi}_1}{ \hat{\chi} + \hat{\chi}_1}, \frac{ - \bar{\chi}_1}{ \hat{\chi} + \hat{\chi}_1} ))
 \end{eqnarray} for any integer $L > 0$, any set of diffeomorphisms $\{ h_1,\cdots, h_L \} \subset \textnormal{\Diff}^{1+\epsilon}(M,\mu)$ which is $\eta-$nontransverse to $E_{1}$,
there exist integers $K, n_0 \geq 1$, real numbers $\kappa_1 > 0, \kappa_2 \in (-\infty, \kappa_1)$  such that IFS  $\{g^{K}h_i  \}_{i=1}^{L}$ is $(n_0, \kappa_1, \kappa_2, b)-$ uniform.
\end{prop}
\begin{proof}
By \eqref{ds ineq 3} and \eqref{ds ineq 32} in Definition \ref{def partially hyperbolic}, we have
\begin{eqnarray*}
\xi : =  \frac{1}{2}\min((\hat{\chi}_1- \bar{\chi}_1 ), - \bar{\chi}_1 )> 0
\end{eqnarray*}
and
\begin{eqnarray*}
\sup_{v \in E_{1}\setminus \{0\}}\frac{\norm{Dg(v)}}{\norm{v}} < e^{\bar{\chi}_1} < e^{\hat{\chi}_1}  < \inf_{u \in E_{2}\setminus \{0\}}\frac{\norm{Dg(u)}}{\norm{u}}.
\end{eqnarray*}
We set \begin{align} \label{A def} A = \sup_{1 \leq i \leq L}( \norm{h_i}_{C^1} , \norm{h_i^{-1}}_{C^1} ). \end{align}

By $\eta-$nontransverse hypothesis and the compactness of $M$, there exists a constant $\tau > 0$ such that
for any $x \in M$, any $E \in Gr(T_{x}M, d-b)$, we have
\begin{align} \label{tau 1}
\#\{ i | \mbox{ $\angle(Dh_{i}(x,E), E_1(h_i(x))) > \tau$}\} > (1-\eta)L.
\end{align}

We claim that there exists $C_{\tau} > 0$ such that for any $x\in M$, any $E \in Gr(T_{x}M, d-b)$ satisfying $\angle(E, E_{1}(x)) > \tau$, for any integer $q  \geq 1$ we have
\begin{align} \label{c 1 tau}
\sup_{v \in U(E^{\perp})}\norm{P_{(Dg^{q}(x,E))^{\perp}}(Dg^{q}(x,v))} \leq C_{\tau}e^{q\bar{\chi}_1}.
\end{align}
Indeed, there exists a unique pair $(u_1,u_2) \in E_1(x) \times E$ such that $v = u_1 + u_2$. By $v \in U(E^{\perp})$ and $\angle(E, E_1(x)) \in (\tau, \frac{\pi}{2}]$, we have $\norm{u_1} = (\sin \angle(u_1,v))^{-1} \leq \frac{\pi}{2}\tau^{-1}$. Then 
\aryst
\norm{P_{(Dg^{q}(x,E))^{\perp}}(Dg^{q}(x,v))} &=& \norm{P_{(Dg^{q}(x,E))^{\perp}}(Dg^{q}(x,u_1))} \\
&\leq& \norm{Dg^{q}(x,u_1)} \leq \frac{\pi}{2}\tau^{-1}e^{q \bar{\chi}_1}.
\earyst
We obtain \eqref{c 1 tau} for $C_{\tau} = \frac{\pi}{2}\tau^{-1}$.

Denote by $h_{(i)} =  g^{K}h_{i}$ for $1 \leq i \leq L$, where $K$ is a large integer whose value will be determined depending solely on $\eta, \xi, C_{\tau}$ and $g$.

Given $l \geq 1$, $(i_s)_{s=1}^{ l} \in \{1,\cdots, L\}^{l}$, we denote
\begin{eqnarray*}
f(i_1, \cdots, i_l) := h_{(i_l)} \cdots h_{(i_1)} = g^{K}h_{i_l} \cdots g^{K}h_{i_1}.
\end{eqnarray*}

For any infinite sequence $\omega = (i_k)_{k\in \N} \in \{1,\cdots, L\}^{\N}$, any $n \geq 1$, we denote
\begin{eqnarray*}
f_{\omega}^{n} &=& f(i_1, \cdots, i_{n}).
\end{eqnarray*}
It is clear that we have
\begin{eqnarray}  \label{f n+1 f n}
f_{\omega}^{n+1} = g^{K}h_{i_{n+1}} f_{\omega}^n.
\end{eqnarray}

To show that for some $K > 0$, the IFS associated with $\{g^{K}h_i | 1\leq i \leq L \}$ is $(n_0, \kappa_1, \kappa_2, b)-$ uniform for some $n_0, \kappa_1 > 0, \kappa_2 \in (-\infty, \kappa_1)$, it is enough to show that: for sufficiently large $K > 0$, there exist $\kappa_1 >0, \kappa_2 \in (-\infty, \kappa_1)$ such that for any $x\in M$, any $E \in Gr(T_{x}M,d-b)$, any $n \geq 1$, we have
\begin{align}
C(x,E,n+1) - C(x,E,n)
 < -\kappa_1,   \label{negative 1}\\
D(x,E,n+1)-D(x,E,n)
> -\kappa_2.  \label{negative 2}  
\end{align}

Given any $K > 0$, we set $\kappa_{1} = K(  - \bar{\chi}_1  - \frac{1}{3}\xi )$.
We will detail the proof of the first inequality, the second one with $\kappa_2 = -K( \hat{\chi}_1 - \frac{1}{3}\xi )$ follows from a similar argument. We clearly have $\kappa_1 > 0$ and $\kappa_2 < \kappa_1$.

By conditioning on the first $n-$iterations, the inequality \eqref{negative 1} is reduced to the following : 
  there exists $K_0 > 0$, such that for any $K \geq K_0$, for any $x \in M$, any $E \in Gr(T_{x}M,d-b)$, any $n \geq 1$, any $(i_k)_{k=1}^{n}$, the following is true :
\begin{eqnarray} 
&& J \nonumber\\
& :=& \EV(\log \sup_{v \in U(E^{\perp})} \norm{P_{Df_{\omega}^{n+1}(x,E)^{\perp}}(Df_{\omega}^{n+1}(x,v))} - \log \sup_{v \in U(E^{\perp})} \nonumber \norm{P_{Df_{\omega}^{n}(x,E)^{\perp}}(Df_{\omega}^{n}(x,v))} \nonumber \\
&& \quad \quad| \omega_k = i_k \mbox{ for $1\leq k \leq n$}) < -\kappa_1. \label{J negative}
\end{eqnarray}

For any $1 \leq j \leq L$, we set
\begin{eqnarray*}
&&J_{j}\\
& := &\EV(\log \sup_{v \in U(E^{\perp})} \norm{P_{Df_{\omega}^{n+1}(x,E)^{\perp}}(Df_{\omega}^{n+1}(x,v))}  - \log \sup_{v \in U(E^{\perp})} \label{J j p} 
 \norm{P_{Df_{\omega}^{n}(x,E)^{\perp}}(Df_{\omega}^{n}(x,v))} \\
&&| \omega_k = i_k , \forall k \in [1, n], \omega_{n+1} = j).  \nonumber
\end{eqnarray*}

Then we have $J = \frac{1}{L}\sum_{j=1}^{L} J_{j}$.

By \eqref{f n+1 f n}, for each $1 \leq j \leq L$   we have
\begin{eqnarray*}
&&J_{j} \\
&\leq& \EV(\log \sup_{v \in U(E^{\perp})} \frac{\norm{P_{Df_{\omega}^{n+1}(x,E)^{\perp}}(Df_{\omega}^{n+1}(x,v))}}{\norm{P_{Df_{\omega}^{n}(x,E)^{\perp}}(Df_{\omega}^{n}(x,v))}}  \nonumber 
  |\omega_k = i_k , \forall k \in [1,n],\omega_{n+1} = j ) \nonumber \\
&\leq&\EV(\log \sup_{v \in U(Df^{n}_{\omega}(x,E)^{\perp})} \norm{P_{Df_{\omega}^{n+1}(x,E)^{\perp}}(Dg^{K} Dh_{j}(f_{\omega}^n(x),v))}  \nonumber |\omega_k = i_k , \forall k \in [1,n], \omega_{n+1} = j). \label{eq 7}
\end{eqnarray*}

For any $x' \in M$, any $E' \in Gr(T_{x'}M, d-b)$,  any  $j \in \{1,\cdots, L\}$, we set
\begin{align} \label{J j p i x E}
J(j;  x',E') = \log \sup_{v \in U(E'^{\perp})}\norm{P_{(Dg^{K}Dh_{j}(x',E'))^{\perp}}(Dg^{K}Dh_{j}(x',v))}.
\end{align}

Then by \eqref{f n+1 f n}, \eqref{eq 7} and the fact that $f_{\omega}^{n}(x)$ and $Df_{\omega}^{n}(x,E)$ depend only on $(\omega_k)_{k=1}^{n}$ but not on $\omega_{n+1}$, we have
\begin{eqnarray} \label{eq 8}
J \leq\sup_{\substack{ x' \in M \\ E' \in Gr(T_{x'}M, d-b)}} \frac{1}{L} \sum_{j=1}^{L}  J(j ; x',E').
\end{eqnarray}

By the hypothesis that the set of diffeomorphisms $\{h_1,\cdots, h_L\}$ is $\eta-$nontransverse to $E_1$,
for any $x' \in M$, any $E' \in Gr(T_{x'}M,d-b)$, there are more than $(1-\eta)L$ indexes $j$ such that $\angle( Dh_j(x',E'), E_{1}(h_j(x'))) > \tau$. For any such $j$, by  \eqref{c 1 tau}, we have
\begin{align} 
&\sup_{v \in U(E'^{\perp})}\norm{P_{(Dg^{K}Dh_j(x',E'))^{\perp}}(Dg^{K}Dh_j(x',v))}  \nonumber \\
&\leq \sup_{u \in U(Dh_j(x',E')^{\perp})} \norm{P_{(Dg^{K}Dh_j(x',E'))^{\perp}}(Dg^{K}(h_j(x'),u))} \sup_{v \in U(E'^{\perp})}\norm{P_{Dh_j(x',E')^{\perp}}(Dh_j(x',v))} \nonumber \\
 &\leq AC_{\tau}e^{K\bar{\chi}_1}. \label{eq 10000}
\end{align}
For all $1 \leq j \leq L$, we have the trivial bound
\begin{align} \label{eq 109}
\sup_{v \in U(E'^{\perp})}\norm{(P_{(Dg^{K}Dh_j(x',E'))^{\perp}}(Dg^{K}Dh_j(x',v))} \leq e^{K\hat{\chi}}A.
\end{align}
Thus by \eqref{eq 109}, \eqref{eq 10000}, \eqref{eq 8} for any $x' \in M$, any $E' \in Gr(T_{x'}M, d-b)$, we have 
\begin{align} \label{eq 110}
J \leq \frac{1}{L}\sum_{j=1}^{L} J( j; x',E' ) \leq \eta K \hat{\chi} + (1-\eta) ( K \bar{\chi}_{1}  + \log C_{\tau} ) + \log A.
\end{align}

By \eqref{eta eta}, we have $\eta < \frac{\xi}{4( \hat{\chi} - \bar{\chi}_1)}$, thus
\begin{eqnarray*}
J &\leq&(\bar{\chi}_1 + \frac{1}{4}\xi) K + (1-\eta) \log C_{\tau} + \log A.
\end{eqnarray*}
 By letting $K_0$ to be sufficiently large depending on $\hat{\chi}_1, \bar{\chi}_{1}, C_{\tau}, A$, for any $K \geq K_0$, we get \eqref{J negative} for
\begin{eqnarray*}
\kappa_1 =  K(-\bar{\chi}_1 - \frac{1}{3}\xi) > 0.
\end{eqnarray*}
This completes the proof.
\end{proof}

\section{Proof of the main theorems}\label{Proof of the main theorems}

\begin{proof}[Proof of Theorem \ref{thm generalization DK} :]
In the \textit{proof} of \cite{DK} Section 10, Corollary 4, the authors showed that under the hypothesis of Theorem \ref{thm generalization DK}, either $\{f_{\alpha}\}$ is $(n_0,\kappa_1, \kappa_2 )-$uniform for some $n_0,\kappa_1 > 0, \kappa_2 \in (-\infty, \kappa_1)$, in which case the transitivity follows from Proposition \ref{prop main criterion}; or $\{f_{\alpha}\}_{\alpha}$ are simultaneously conjugate to rotations close to the original ones, in which case we have ergodicity ( see \cite{DK} ), and hence transitivity.

More precisely, by letting $r=j=d$ in (45) in the proof of Corollary 4 (a) in \cite{DK}, we verify the $(n_0, \kappa_1, \kappa_2,b)-$uniformity condition for $b=1$,  $\kappa_1 = \lambda_1 - \varepsilon$, $\kappa_2 = \frac{d-3}{d-1}\lambda_1 + \varepsilon$, sufficiently small $\varepsilon$ and sufficiently large $n_0$.
\end{proof}

\begin{proof}[Proof of Theorem \ref{thm transitive group} :]Let $TM = E_1 \oplus E_2$ be the uniformly dominated splitting given by the theorem, with $\dim(E_1 ) = d-\ell$ for some integer $\ell \in \{1,\cdots, d-1\}$. Let  $\bar{\chi}_1,\hat{\chi}_1$ satisfy \eqref{ds ineq 3} and \eqref{ds ineq 32} for $g$ in place of $f$, and let $\hat{\chi} > \log \norm{g}_{C^1}$.

We set
$$d_ {\ell} := \dim Gr(M,\ell) = d+ \ell(d-\ell).$$
To simplify notations, we denote the metrics $d_{M}, d_{Gr(M,\ell)}$ both by $d$. This should cause no confusions.

Without loss of generality, we assume that the map $x \mapsto E_1(x)$ is $\theta-$H\"older, for otherwise we can consider $g^{-1}$ instead of $g$.

By Proposition \ref{prop main criterion}, it suffices to show that there exists an integer $L_0 > 0$ such that for any $L \geq L_0$, for any $L$-tuple $(f_1,\cdots, f_L)$ in  a $C^r$ dense subset of $\Diff^{r}(M,\mu)^L$, $\{f_i\}_{i=1}^{L}$ is $\eta-$nontransverse, where $\eta$ is any constant given by Proposition \ref{prop unif to trans} with $g$ given by the theorem. To prove this, we first construct a finite set of divergence-free vector fields $\{V_{\alpha}\}_{\alpha \in \cal A}$ on $M$ such that the following is true.
There exists a constant $\kappa > 0$, such that for any $(x, E) \in Gr(M, \ell)$, there exists a subset $\{\alpha_1, \cdots, \alpha_{d_{\ell}}\} \subset \cal A$, such that
\begin{eqnarray} \label{construct V} 
|\det[(V_{\alpha_i}(x), DV_{\alpha_i}(x,\cdot | E))_{i=1}^{d_{\ell}}]| > \kappa.
\end{eqnarray}
For the notation $DV_{\alpha_i}(x,\cdot | E)$ we recall \eqref{eq 9022} in Definition \ref{lift vector field}.
Such $\{V_{\alpha}\}_{\alpha \in \cA}$ exists by $d \geq 2$.

For any $B = (B^{(i)})_{i=1}^{ L} \in \R^{L|\cal A|}$, where $B^{(i)} = (B^{(i)}_{\alpha})_{\alpha \in \cA} \in \R^{|\cal A|}$, we define
\begin{eqnarray*}
V^{(i)}(B,x) = \sum_{\alpha \in \cal A} B^{(i)}_{\alpha}V_{\alpha}(x).
\end{eqnarray*}

We denote by $H^{(i)}_{\alpha} = (B^{(j)})_{1 \leq j \leq L} \in \R^{L |\cal A|}$ where $B^{(j)} = \delta_{j=i} (\delta_{\alpha = \beta})_{\beta \in \cA}  \in \R^{\cal A}$ for any $1 \leq  j \leq L$. Here $\delta_{j=i}, \delta_{\alpha = \beta}$ are Kronecker delta functions.  For each $B \in \R^{L|\cal{A}|}$, we define a $L-$tuple denoted by $(f_i(B))_{i=1}^{L} \subset \textnormal{\Diff}^{r}(M,\mu)$ as follows,
\begin{eqnarray*}
f_i(B,x)= \Psi(V^{(i)}(B,\cdot), 1)f_i(x), \quad 1 \leq i \leq L.
\end{eqnarray*}
Recall that $\Psi$ is defined  in Definition \ref{lift vector field}.

By definition, we know that $f_i(0,\cdot) = f_i, \forall 1 \leq i \leq L$.
We will show that there exists arbitrarily small $B \in \R^{L|\cal A|}$ such that $\{f_i(B,\cdot)\}_{i=1}^{L}$ is $\eta-$nontransverse.

For any $(x,E) \in Gr(M, \ell)$, we can construct a map as follows.
\begin{eqnarray*} 
\Phi_{x, E} : \cal D_1 &\to& Gr(M,\ell)^{L} \\
 \Phi_{x,E}(B)&=&  (\GG(f_{i}(B,\cdot))(x,E))_{i=1}^{ L}.
\end{eqnarray*}
Here $\cal D_1$ denotes the unit disk in $\R^{L|\cal A|}$.

Take any $(y,F) \in Gr(M,\ell)$, for any $1 \leq i \leq L$, let $\{\alpha^{(i)}_{1},\cdots, \alpha^{(i)}_{d_{\ell}}\} \subset \cal A$ be the indexes satisfying \eqref{construct V} for $(x,E) = \GG(f_i)(y,F)$.
By differentiating $\Phi_{y,F}$ at the origin, we obtain
\begin{eqnarray} \label{der phi 1}
D_{\alpha^{(i)}_{j}}\Phi_{y,F}(0) = ((V_{\alpha^{(i)}_{j}}(x), DV_{\alpha^{(i)}_j}(x,\cdot |E)\delta_{i=k})_{ k= 1}^{ L}.
\end{eqnarray}
We set $P = \sum_{1 \leq i \leq L, 1 \leq j \leq d_{\ell}} \R H^{(i)}_{\alpha^{(i)}_j}$. Then by \eqref{der phi 1} \eqref{construct V}, we can see that
\begin{eqnarray*}
| \det(D\Phi_{y,F}(0) : P \to \prod_{k=1}^{L}T_{\GG(f_k)(y,F)}Gr(M,\ell)) | > \kappa^{L}.
\end{eqnarray*}
Since $r \geq 2$, $\Phi_{x,E}(B)$ is $C^1$ in $x,E,B$.
By equicontinuity of $D\Phi_{y,F}(B)$ in $y,F,B$, we see that there exists a constant $\epsilon > 0$, $\kappa_0 > 0$ such that $\Phi_{y,F} : B(0,\epsilon) \to Gr(M, \ell)^L$ is a submersion for all $(y,F) \in Gr(M,\ell)$, and  for any $(B, y, F) \in B(0,\epsilon) \times Gr(M,\ell)$, there exists a linear subspace $P' \subset \R^{L|\cal A|}$ of dimension $Ld_{\ell}$,  such that
\begin{eqnarray} \label{lower det}
|\det(D\Phi_{y,F}(B)|_{P'})|> \kappa_0.
\end{eqnarray}
Let $\cD = B(0,\epsilon)$.

We define the following subsets of $Gr(M,\ell)$ and $Gr(M,\ell)^L$.
\begin{eqnarray*}
\Sigma_0 &=& \{(x,E)\in Gr(M,\ell) | E \mbox{ is not transverse to }E_1(x) \}, \\
\Sigma &=& \{(x^i,E^i)_{i=1}^{ L} \in Gr(M,\ell)^L | \#\{i | (x^i,E^i) \in \Sigma_0 \} \geq \eta L \}.
\end{eqnarray*}

Then for any $B \in \cal D$, $\{f_{i}(B,\cdot)\}_{i=1}^{ L}$ is $\eta-$nontransverse if and only if
\begin{eqnarray}\label{nontrans equiv cond}
\Phi_{x,E}(B) \notin \Sigma, \quad \forall (x,E) \in Gr(M,\ell).
\end{eqnarray}

Let $\beta > 0$ and let $\delta > 0$ be a sufficiently small real number, we let $\cal N_{\delta}$ be a  $\delta^{1+\beta}-$net in $Gr(M,\ell)$. Then we have $\cal N_{\delta} < \delta^{-d_{\ell}(1+\beta) - \beta}$ for sufficiently small $\delta$.
We set
\begin{eqnarray}\label{def omega delta}
\Omega_{\delta} = \cD \bigcap (\bigcup_{(x,E) \in \cal N_{\delta}} (\Phi_{x,E})^{-1}(\Sigma_{\delta})),
\end{eqnarray}
where $\Sigma_{\delta}$ is the $\delta-$neighbourhood of $\Sigma$ in $Gr(M,\ell)^L$.
We claim that for any sufficiently large $L$ and any sufficiently small $\beta$ the following holds:

(1) The measure of $\Omega_{\delta}$ tends to $0$ as $\delta$ tends to $0$;

(2) For any sufficiently small $\delta$, any $B \in \cal D \setminus \Omega_{\delta}$, we have \eqref{nontrans equiv cond}.

This will conclude the proof since for any $\varepsilon > 0$, by taking $\delta > 0$ sufficiently small, we can find $B \in \cD \setminus \Omega_\delta$ such that $d_{C^r}(f_i(B,\cdot), f_i) < \varepsilon$ for all $1 \leq i \leq L$. By (2), $\{f_{i}(B,\cdot)\}_{i=1}^{L}$ is $\eta-$nontransverse with respect to $E_1$. Since $V(B,\cdot)$ is divergence free, $f_{i}(B,\cdot) \in \diff^{r}(M, \mu)$ for all $1 \leq i \leq L$. Then by Proposition \ref{prop main criterion}, the IFS $\{ g, f_{1}(B,\cdot),\cdots, f_{L}(B,\cdot)\}$ is transitive.

To see (1), we first show the following lemma.
\begin{lemma}
We have $
 HD(\Sigma_0) < d_{\ell}.$
\end{lemma}
\begin{proof}
Recall that  the map $x \mapsto E_1(x)$ is $\theta-$H\"older.
Take any $\beta' \in (\theta^{-1} , \frac{d+1}{d})$. For small $\delta > 0$, we choose a $\delta^{\beta'}-$net in $M$, denoted by $\cal{M}$ such that $\#\cal{M} = O(\delta^{-d\beta'})$. For each $x \in \cal{M}$, the subset of $Gr(T_{x}M, \ell)$ defined as $A(x) = \{ E\in Gr(T_{x}M, \ell) | E \mbox{ is not transverse to }E_{1}(x)   \}$ has dimension $\ell(d-\ell)-1$. Thus $A(x)$ can be covered by $O(\delta^{-\ell(d-\ell)+1})$ many $\delta-$balls. For each $y \in M$, there exists $x \in \cal{M}$ such that $d(x,y) < \delta^{\beta'}$, thus $d((x,E_1(x)), (y,E_1(y))) \lesssim\delta^{\theta \beta'}$. Since $\theta \beta' > 1$, when $\delta'$ is sufficiently small, we can choose $\cal{E}$, a subset of $\bigcup_{x \in \cal{M}} A(x)$, such that : 1. $\# \cal{E}  = O( \delta^{-d\beta' - \ell(d-\ell) +1})$; 2. $\cal{E}$ forms a $2\delta-$net in $\Sigma_0$. By the choice of $\beta'$, we see that $d \beta' -1 + \ell(d-\ell) < d + \ell(d-\ell) = d_{\ell}$. Since $\delta$ can be arbitrarily small, this implies that $HD(\Sigma_0) < d_l$.
\end{proof}

Let $\xi > 0$ be a real number such that $HD(\Sigma_0) < d_{\ell} - \xi$.
Then we have
\begin{eqnarray*}
HD(\Sigma) \leq \eta L HD(\Sigma_0) + (L - \eta L) d_{\ell} < L d_{\ell} - \eta L \xi.
\end{eqnarray*}
Take an arbitrary constant $\beta > 0$. Then it is a classical fact that there exists a constant $C_1 = C_1(\Sigma, \beta) > 0$ such that the volume of $\Sigma_{\delta}$ satisfies
\begin{eqnarray*}
|\Sigma_{\delta}| < C_1\delta^{ \eta L \xi - \beta}.
\end{eqnarray*}
Then by \eqref{lower det}, we see that there exists $C_2 =C_2(C_1,\kappa_0)$ such that,
\begin{eqnarray*}
|(\Phi_{x,E})^{-1}(\Sigma_{\delta})| < C_2 \delta^{\eta L \xi - \beta}.
\end{eqnarray*}
Then by \eqref{def omega delta}, there exists $C_3 = C_3(C_2,\beta)$ such that
\begin{eqnarray}\label{def omega deltafff}
|\Omega_{\delta}| < C_3 \delta^{- d_{\ell}(1+\beta) - \beta + \eta L \xi - \beta}
\end{eqnarray}
Then by letting $L > \eta^{-1}\xi^{-1}d_{l}$ and $\beta$ sufficiently small, we can see that $|\Omega_{\delta}| \to 0$ as $\delta \to 0$. This proves (1).

To see (2), we take an abitrary $(x,E) \in Gr(M,\ell)$ and $B \in \cD \setminus \Omega_{\delta}$. Then there exists $(y, F) \in \cal N_{\delta}$ such that $$d((x,E), (y,F)) < \delta^{1+\beta}.$$
Then by $B \notin \Omega_{\delta}$, we see that $\Phi_{y,F}(B) \notin \Sigma_{\delta}$.
We notice that by $r \geq 2$, the function $\GG(f_i(B, \cdot)) : Gr(M,\ell) \to Gr(M,\ell)$ is $C^1$ with $C^1$ norm uniformly bounded for all $1 \leq i \leq L$ and $B \in \cal D$. Thus there exists $C > 0$ independent of the choices of $(x,E), (y,F)$ such that
$$d(\Phi_{x,E}(B), \Phi_{y,F}(B)) < C d((x,E), (y,F)) <  C \delta^{1+\beta}, \quad \forall B \in \cal D.$$
This implies that for all sufficiently small $\delta > 0$, we have
$$d(\Phi_{x,E}(B), \Sigma) > d(\Phi_{y,F}(B), \Sigma) - d(\Phi_{x,E}(B), \Phi_{y,F}(B)) \geq \delta- C\delta^{1+\beta} > 0.$$
This concludes the proof of (2), and thus concludes the proof of Theorem \ref{thm transitive group}.

\end{proof}

\
\

{\footnotesize \noindent Zhiyuan Zhang\\
Institut de Math\'{e}matique de Jussieu---Paris Rive Gauche, B\^{a}timent Sophie Germain, Bureau 652\\
75205 PARIS CEDEX 13, FRANCE\\
Email address: zzzhangzhiyuan@gmail.com

\end{document}